\documentclass[paper,12pt]{amsart}
\usepackage{subfigure}
\usepackage{amsfonts}
\usepackage{newlfont}
\usepackage[centertags]{amsmath}
\usepackage{amsfonts}
\usepackage{amssymb}
\usepackage{graphicx}
\usepackage{pb-diagram}
\usepackage{footnote}
\usepackage{mathrsfs}
\usepackage{makecell}
\usepackage{enumerate}
\usepackage{tikz}
\usepackage{tikz-3dplot}
\usepackage{float}
\usepackage[all]{xy}
\usetikzlibrary{calc}
\usepackage{tikz-cd}
\usetikzlibrary{cd}
\usepackage{tcolorbox, empheq}

\numberwithin{equation}{section}
\newtheorem{theorem}{Theorem}[section]
\newtheorem{definition}[theorem]{Definition}
\newtheorem{proposition}[theorem]{Proposition}

\newtheorem{lemma}[theorem]{Lemma}
\newtheorem{corollary}[theorem]{Corollary}
\newtheorem{rmk}[theorem]{Remark}

\newtheorem{example}[theorem]{Example}

\newcommand{\R}{\mathbb R}

\newcommand{\lB}  {^B\!\nabla}
\newcommand{\lF}  {^F\!\nabla}
\newcommand{\lf}  {^f\!\nabla}
\newcommand{\TB}  {^B\!T}
\newcommand{\TF}  {^F\!T}
\newcommand{\Tf}  {^f\!T}
\newcommand{\cc}    {\nabla}
\newcommand{\dc}    {\nabla^*}
\usepackage{hyperref}
\hypersetup{
    colorlinks=true,
    linkcolor=blue,
    filecolor=magenta,      
    urlcolor=cyan,
    pdfpagemode=FullScreen,
    }

\title[Warped product on information geometry]{On warped product on information geometry: statistical manifolds and statistical models} 
\author{Nicolas Martinez-Alba}
\address{Departamento de Matem\'aticas, Universidad Nacional de Colombia, Bogot\'a, Colombia}
\email{nmartineza@unal.edu.co}

\author{Olga Garatejo Escobar}
\address{Universidad Nacional de Colombia, Bogot\'a, Colombia}\email{ocgaratejoe@unal.edu.co}
\begin{document}


\begin{abstract}
We study a differential geometric construction, the warped product, on the background geometry for information theory. Divergences, dual structures and symmetric 3-tensor are studied under this construction, and we show that warped product of manifolds endow with such structure also is endowed with the same geometric notion. However, warped product does not preserve canonical divergences, which in particular shows that warped product lacks of meaning in the information theory setting.

\end{abstract}

\maketitle

\tableofcontents
\section{Introduction}

%The novelty: In the literature of this topic there are some results that are not complete clear, and for this manuscript we aim to collect results in for further consult, in particular our presentation give an improvement on the results of statistical inmersion and the global proofs of classical results that are proved in terms of Christoffle symbol.

%The goal of this note is to complete the scenario of the warped product (the dual structures and the Amari-Chentsov tensor) under the geometric construction of warped product. 

  Science information is a term used to describe the interdisciplinary studies to explore and to analyze information theory and data scenarios \cite{science}. This type of studies lie on the interaction of several areas, such as statistical inference, signal processing, machine learning, or neural networks. A fundamental part of this interaction is the role of statistics, where probability theory is essential for making sense of data and modeling uncertainty. However, the scope of science information extends beyond statistics, and new developments have emerged in other branches of mathematics. One notable example of this interaction is the so called {\it geometric science information}. %which applies geometrical methods to enhance the understanding and organization of data. This field represents a convergence of diverse scientific disciplines aimed at optimizing data interpretation and knowledge extraction.}

    Information geometry is a specialized field within geometry science information that applies concepts from differential geometry to the study of statistical models. Its foundation goes back to the independent works of H. Hotelling (1930) and R. Rao (1945), who proposed a mathematical framework to give a new point of view to families of probability  distributions. A central concept in this field is the Fisher information metric, which provides a differentiable structure on the space of probability densities. This metric allows for the exploration of various geometric properties within statistical models, such as distances between probability distributions and their curvature. The Fisher metric is deeply connected to key statistical concepts, including expected value, entropy, and divergence, making it a powerful tool for understanding the geometry of information theory. For a more detailed discussion of this topic, we refeer to Nielsen's work on information geometry \cite{frank} and Amari and Nagaoka's \cite{amariMeth} comprehensive insights.

In the probability setting, the exponential and mixture families are two fundamental statistical concepts that describe two types of probability densities, each with its own statistical and algebraic characteristics. These families gain importance in information geometry, where they not only provide statistical properties but also induce geometric structures, such as dual connections, torsion, geodesics, curvatures, divergences, and symmetric 3-tensors. In this work we focus our result to dual connection, Amari-Chentsov tensor and divergence functions. {\it Divergence functions} are generating functions that quantify the distance between two probability distributions, playing a crucial role in understanding of their geometric relations. The notion of {\it dual structure}, introduced by Amari and Chentsov, builds on the relationship between the exponential and mixture families, where each family is associated with a specific affine connection. This duality allows the analysis of statistical models from two complementary perspectives, providing deeper insights into their geometry. The {\it Amari-Chentsov tensor} together with divergences, offers a comprehensive framework for studying both, the statistical and geometric, properties of probability distributions, bridging the gap between statistical inference and geometry. The relationship between divergence, dual structure, and the Amari-Chentsov tensor lies at the heart of information geometry.

In this geometrical setting, the interaction of these three geometrical notion can be seen, in the geometry and in the probability world, as the following 2-level diagram of dependency:\\

%$$\begin{diagram}
%%\node{} \node{\mbox{divergence}}\arrow{sw,t} \arrow{se,t} \node{}
%\node{A} 
%\node{\mbox{canonical divergence}}\arrow{s} %\arrow{es,t} \node{}\\
%%\node{\mbox{dual}} \node{equivalent to}\node{3-\mbox{tensor}} 
%\node{B} \node{(e,m)-\mbox{dual} \Leftrightarrow \mbox{Amari-Chentsov tensor}}\\
%
%\end{diagram}$$

%$$\begin{diagram}
%\node{\mathrm{Divergence}}\arrow{e}%\arrow{se} 
%\node{\mathrm{dual\ connection}\Leftrightarrow \mathrm{3-Tensor}}\\
%\node{\mathrm{Canonical\ divergence}}\arrow{e} 
%\node{(e,m)-\mathrm{dual}\Leftrightarrow \mathrm{Amari-Chentsov\ Tensor}}\\ 
%\end{diagram}$$

\begin{tikzcd}[column sep=0.0001cm]
       &  \mathrm{Divergence} \arrow{dl} \arrow{dr}  & \\
   \mathrm{dual\ connection} \arrow[leftrightarrow]{rr}   &          &  \mathrm{\tiny 3-tensor}
\end{tikzcd}  
\begin{tikzcd}[column sep=0.0001 cm]
       & \mathrm{Canonical\ divergence}  \arrow{dl} \arrow{dr}  & \\
  (e,m)-\mathrm{dual} \arrow[leftrightarrow]{rr}   &    &  {\tiny \mathrm{ Amari-Chentsov}}
\end{tikzcd}\\

where the left-hand side diagram is the geometrical model of the information geometry concepts in the right-hand side diagram.
%This document is based on the differential geometric study of some linear connections and symmetric tensors in the case of statistical manifolds and statistical models. Statistical model $(M, g)$ is a differential manifold $M$ embedded in the space of strictly positive measurements on a set of indices $I$ and $g$ the pull-back by the Fisher metric embedding. Due to the structure of the space of strictly positive measurements, two torsion-free linear connections associated with the exponential families and mixture families satisfy the Leibniz-type relationship concerning the Fisher metric which in turn comes equipped with two dual connections with respect to the Fisher metric, respectively for the exponential and mixture family, producing two types of statistical models. Generalizing, Lauritzen in \cite{lau} defines a statistical manifold or statistical structure $(M, g, T)$ by a metric and a symmetric 3-tensor called the Amari-Chentsov tensor, which in turn is determined by two dual connections to each other and free of torsion. Note that the difference between the two is the existence of statistical significance in the statistical manifolds, however, there is a representation result, in which every statistical manifold (geometric information) is realized as a statistical model (statistical information). 

The warped product of Riemannian manifolds is a technique that allows to modify lenghts (by conformal symmetries of a metric) but preserving angles on each factor and trivial perpendicularity between both factors of the product manifold. This structure is a remarkable construction within the metric geometry \cite{oneil} but also in the context of information geoemtry as in \cite{SB,Fuj}. This paper is devoted to the study of a special case of geometric structure, the warped product of Riemannian manifolds, but in the context of statistical models and statistical manifolds. The situation of warped product for the case of flat dual connection is already known \cite{leonard}. But in the literature there is a lack of an associated result of this metric product to symmetric 3-tensor or the divergences functions. The main goal of this paper is to complete the previous diagram under the existence of a warped product. In particular, we have proved the following theorem:

% in particular, we use this notion to prove that the warped product preserves statistical manifolds but not statistical models, Therefore the representation mentioned above is not an isomorphism in structures, that is to say that although every statistical variety can be immersed with a statistical model, this immersion is not real inclusion.

%This construction is carried out in the following way: the warped product of two statistical manifolds (in dual connection or 3-tensor version) remains a statistical manifold, however, the warped product of two statistical models is not a statistical model, since the Fisher metric can be described using canonical divergence (generated through geodesics), but in the warped product the product of geodesics is not a geodesic.

{\bf Main Theorem:} {\it 
Let $F$ and $B$ be two manifolds, each one endowed with torsion-free dual connections $(\nabla_F,\nabla_F^*)$, $(\nabla_B,\nabla_B^*)$ and symmetric 3-tensors $T_F$, $T_B$.  Given a nonvanishing function $f$ in $B$,  then the warped product $F\times_f B$
\begin{itemize}
\item[(a)] is endowed with a trosion-free dual connections $(\nabla_f,\nabla_f^*)$ which is lifted from  $(\nabla_F,\nabla_F^*)$ and $(\nabla_B,\nabla_B^*)$.
\item[(b)] is endowed with a symmetric 3-tensor $T_f$ which is  lifted from $T_F$ and $T_B$.
\end{itemize}

If in addition there exist divergence functions $D_F$ and $D_B$ for the structures $(\nabla_F,\nabla_F^*, T_F)$ and $(\nabla_B,\nabla_B^*, T_B)$, then the warped product $F\times_f B$.
\begin{itemize}
\item[(c)]  is endowed with a divergence function $D_f$ for the structure $(\nabla_f,\nabla_f^*,T_f)$, which is  lifted from $D_F$ and $D_B$.
\end{itemize}}

Though the warped product preserves the main ingredients of the background geometry for information theory, this product is useless in information theory. When we restrict the main theorem to a {\it canonical divergence}, which contains more information than divergence, we realize that canonical divergence is not preserved by warped product. In conclusion, warped product is a crucial difference between the geometrical model of information theory and the actual information geometry.

{\bf Organization of the paper:} Section 2 is a brief summary of the explicit description of the Riemannian geometry used in the following sections, including the main geometric construction, the {\it warped product} of Riemanninan metrics. Section 3, contains the main notions of the geometry of statistical model, including the Fisher metric, exponential and mixture families and torsion-free dual  connections with explicit examples. In section 4, we abstract the notion of statistical model to any Riemannian manifold summarizing classical results. In particular we present new proofs for already known facts on information geometry
(Theorem \ref{estructuradualconexiAlpha} and Corollary \ref{cor4.5}) avoiding the use of Christoffel symbols (which are the usual proofs in literature). Also, in this Section we present an improved result for isostatistical immersion of statistical manifolds, Proposition~\ref{prop:nueva1}. In the last section we give the proof of our main theorem. The first item of the {\bf Main Theorem} is known in the literature \cite{leonard}, and in Section~5 we give the  proof of the remaining claims. The second and third claims are stated and proved in Theorem \ref{estructuraestadisticaWarped}, Lemma \ref{lem5.4}, and Theorem \ref{thm:warpedDiv}. Finally, we conclude that the warped product does not work in the case of canonical divergence by using the local representation of canonical divergence.

% statistical spaces, including in particular the linear connections associated with probability densities, known as e-connection and m-connection. Chapter 2, is dedicated to the generalization of statistical models in Riemannian geometry known as statistical varieties or dual structures. In this chapter we present common results in the literature and, Although not all of them are proof, the demonstrations we present are original and are more accessible to a first reading on this research topic. The notion of isostatistical immersion is also introduced to state the representation theorem of the statistical manifold with which we present the first new result, Proposition 2.5.1, which is about how the statistical immersion preserves Fisher metrics and Amari-Chentsov tensors. Finally, chapter 3, is about the warped product and presents known results. but for the completeness of the document we rewritten some in greater detail, In particular, section 3.3 presents new results (theorem 3.3.1) for the warped product of the Amari-Chentsov tensor which are not known in the literature and completes the results in this theory. We finish the chapter by making explicit the difference between statistical models and statistical manifolds using divergences.

\section{Geometrical setting}

Let us begin by presenting some basic concepts of riemann manifolds. We refer the reader to \cite{Jo} for a more detailed description.

\subsection{Basics on Riemann geometry}
 A manifold $M$ of dimension $n$ is a connected paracompact Hausdorff space for which every point $x\in M$ has a neighborhood $U_x$ that is homeomorphic to an open subset $\Omega_x$ of $\R^n$. Such a homeomorphism $\phi_x : U_x \to \Omega_x$ is called a {\it coordinate chart}. If for two charts the function $\phi_x\circ \phi_y^{-1}:\Omega_y\to \Omega_x$ is a $C^r$-diffeomorphism we say that the manifold has a $C^r$-differentiable structure. The  collection $\{(U_x,\phi_x)\}$ is called {\it differentiable structure} or {\it smooth structure} for $M$.  We denote $T_xM$ to the vector space which consists of all tangent vectors to curves in $M$ on the point $x$. It is called the \emph{tangent space} of $M$ at the point $x$. 

A {\it Riemannian metric} on a differentiable manifold $M$ is given by a inner product $g$ on each tangent space $T_x M$ which depends smoothly on the base point $x$. A {\it Riemannian manifold} is a differentiable manifold equipped with a Riemannian metric. In any system of local coordinates $(x_1,\ldots,x_n)$ from coordinates charts, the Riemannian metric is represented by a positive definite, symmetric matrix $(g_{ij}(x))_{1\leq i,j\leq n}$ where the coefficients depend smoothly on $x$.  

Let $\gamma : [a, b]\to  M$ be a smooth curve. The length of $\gamma$ is defined as:

$$L(\gamma)=\int_a^b \|\gamma'(t)\|dt,$$ 
where the norm of the tangent vector $\gamma'$ is given by the Riemannian metric as $\| \gamma'(t)\|=g_{\gamma(t)}(\gamma'(t),\gamma'(t))$. This value is invariant under re-parametrization of the curve. Taking the infimum of the values $L(\gamma)$ among all the curves $\gamma$ joining two points $p,q\in M$ we can define a distance function on $M$ and the topology of this distance coincides with the topology of the manifold structure of $M$.

The metric tensor $g$ also allows us to define a natural differential operation on vector fields\footnote{Recall that {\it vector fields} are the same as differential operator of order 1 and can be represented as a smooth section from $M$ to $TM$. The space of vector fields is denoted by  $\Gamma(TM)$} that extend the notion of directional derivatives in the Euclidean case. This is known as {\it Levi-Civita connection}  $\nabla^{(0)}:\Gamma(TM)\times \Gamma(TM)\to \Gamma(TM)$ that is $\mathbb{R}$-bilinear but for the algebra of $C ^\infty(M)$ it is tensorial in the first variable and satisfies Leibniz rule for the second variable , i.e., 
$$\nabla^{(0)}_{fX}Y=f\nabla^{(0)}_XY \mbox{\ and\ } \nabla^{(0)}_XfY=(Xf)Y+f\nabla^{(0)}_XY,$$
where $f\in C^\infty(M)$. A fundamental theorem of Riemannian geometry states that this is the {\it unique} connection that is  torsion free and metric, that is:
$$\cc^{(0)}_XY-\cc^{(0)}_YX=[X,Y]\quad \mbox{\ and\  }\quad  Zg(X,Y)=g(\cc^{(0)}_ZX,Y)+g(X,\cc^{(0)}_ZY),$$
for any three vector fields $X,Y,Z$ in $M$, and $[\cdot,\cdot]$ is the commutator of vector fields as first-order differential operator, i.e., $[X,Y]=X\circ Y-Y\circ X$.

This operation also can be used for the tangent of a curve (also interpreted as the vector field along a curve) and gives us the following situation: given a smooth curve $\gamma:(a,b)\to M$, the curve is called {\it geodesic} if it satisfies $\nabla_{\gamma'}\gamma'=0$. In a local coordinates $(x_1,\dots,x_m)$, the geodesics can be written for each $i=1,\dots,m$ as the second-order ODE:
$$x_i''(t)+\sum_{j,k}\Gamma^i_{jk}(\gamma(t))x_j'(t)x_k'(t)=0,$$
where the functions $\Gamma^i_{jk}$ are known as the {\it Christoffel symbols} of $\nabla$.
\begin{theorem}{\cite[Theorem~1.4.2]{Jo}}
Let $M$ be a Riemannian manifold, $x\in M$ and $v\in T_xM$. Then there exist $\epsilon>0$ and precisely one geodesic $c : [0, \epsilon] \to M$ with $c(0) = x, c'(0)= v$. In addition, $c$ depends smoothly on $x$ and $v$.
\end{theorem}

%Finally, we want to mention some of the local characteristics of Riemannian metrics, that induce classifications of such structures, the {\it curvature}. From the Levi-Civita connection it is possible to define the following curvature tensors, Riemann curvature and sectional curvature (respectively):
%\begin{align*}
%R(X,Y)&=\nabla_X\nabla_Y-\nabla_Y\nabla_X-\nabla_{[X,Y]} \\
%R_s(X,Y)&=\dfrac{g(R(X,Y)Y,X)}{g(X,X)g(Y,Y)-g(X,Y)^2} .
%\end{align*}
%where the operator $[X,Y]$ is the commutator of vector fields as the differential operator of order 1.

\subsection{Warped product of Riemannian manifolds}
%\label{sec:warped}

In some circumstances, we want to modify length but not angles, this can be done by conformal metric. When we want to mimic this idea for a product manifold, we include a conformal product in one of the Riemannian factor. This construction is called warped product and we will give some basic results of this notion.
\begin{definition}
Let $(B, g_B)$ and $(F, g_F)$ be two Riemannian manifolds of finite dimension and $f\in C^{\infty}(B)$ a positive function on B. \textbf{The warped product} $M=B\times_{f} F$ with warping function $f$, is the finite dimensional manifold $B\times F$ endowed with the metric $g_{f}$ given by:
\begin{equation*}
    g_{f}=\pi^{*}(g_{B})+\tilde{f}^{2}\sigma^{*}(g_{F}),
\end{equation*}

where $\pi^{*}$ and $\sigma^{*}$ are the pull-backs of the projections $\pi$ and $\sigma$ of $B \times F$ on B and F respectively, and $\tilde{f}=f\circ \pi$.
\end{definition}

We describe briefly the geometry of $M=B\times_{f} F$ in terms of the warping function $f$ and the geometries of B and F. For it, $B$ is called the leaf of $M$, and $F$ is the fiber of $M$. The fibre $\{p\} \times F=\pi^{-1}(p)$ and the leaves $B\times \{q\}=\sigma^{-1}(q)$ are Riemannian submanifolds of $M$ characterized by:

\begin{itemize}
    \item [1.] For each $q\in F$, the map $\pi |_{B\times \{q\}}$ is an isometry onto $B$.
    \item [2.] For each $p\in B$, the map $\sigma|_{\{p\}\times F}$ is a positive homothecy onto $F$, with scale factor $1/f(p)$.
    \item[3.] For each $(p,q)\in M$, the leaf $B\times \{q\}$ and the fiber $\{p\}\times F$ are orthogonal at $(p,q)$.
\end{itemize}
Tangent vectors  to leaves are called horizontal, and those tangent to fibers are called vertical. We denote by $H$ the orthogonal projection of $T_{(p,q)}M$ onto its horizontal subspace $T_{(p,q)}(B\times \{q\})$ and by $V$ the projection onto the vertical subspace $T_{(p,q)}(\{p\}\times F)$. Both projections induce the notion of lifting, to $M$ as follows: if $X\in T_{p}(B)$ and $q\in F$ then the lift $\widetilde{X}$ of $X$ to $(p,q)$ is the unique vector in $T_{(p,q)}(B)$ such that $d\pi(\widetilde{X})=X$ and $d\sigma(\widetilde{X})=0$. If $X \in \mathfrak{X}(B)$ the lift  $\widetilde{X}$ to $B \times F$ is the vector field $\widetilde{X}$ whose value at each $(p,q)$ is the lift of $X_{p}$ to  $(p,q)$. Coordinate systems on the product manifold shows that $\widetilde{X}$ is smooth. %Thus the lift of $X \in \mathfrak{X}(B)$ to $(B \times F)$ is the unique element of $\mathfrak{X}(B\times F)$ that is $X\sim_\pi \widetilde{X}$ and $X\sim_\sigma 0$. 
The set of all such horizontal lifts $\widetilde{X}$ is denoted by $\mathfrak{L}(B)$. Functions, tangent vectors, vector fields, and tensors on $F$ can also be lifted to $B \times F$ in the same way using the projection $\sigma$. In this way, we also get  vertical lifts of vectors fields in $B\times F$, denoted by $\mathfrak{L}(F)$. In particular, if $V \in \mathfrak{X}(F)$ the lifting is the unique $\widetilde{V}$ such that $d\pi(\widetilde{V})=0$ and $d\sigma(\widetilde{V})=V$. Additionally, we denote the horizontal lift on $B \times F$ of a vector field $X \in TB$ by $X^{H}$, and the vertical lift on $B \times F$ of a vector field $U \in TF$ by $U^{V}$.\\

In particular, as a consequence of the liftings we have the properties on the commutator:

\begin{align*}
    [X^{H}, Y^{H}]&=[X, Y]^{H} \in \mathfrak{L}(B)\\
    [U^{V}, V^{V}]&=[U, V]^{V} \in \mathfrak{L}(F)\\
     [X^{H},V^{V}]& =0.
\end{align*}

anoher important notion in this geoemtry, is the {\bf gradient vector field} of a function $\phi$. This is the unique vector field $\mathrm{grad}(\phi)$ such that $g(\mathrm{grd}(\phi),\cdot)=d\phi.$ Liftings of gradient vector fields  are
and on the gradient vector field:
\begin{lemma}\cite[Lemma 34]{oneil}
If $h\in C^{\infty}(B)$, then the gradient of the lift $h \circ \pi$ of $h$ to $M=B \times_{f} F$ is the lift to $M$ of the gradient of $h$ on $B$. In addition we get the following relation:
$$\mathrm{grad} (\phi)^H +\frac{1}{\tilde{f}^2}\mathrm{grad}(\psi)^V=\mathrm{grad}(\pi^*\phi+\sigma^*\psi)$$
for any $\phi\in C^\infty (B)$ and $\psi\in C^\infty (F)$.
\end{lemma}

%The relation of a warped product to the fiber $F$ often involves the warping function $f$. Details of proof are found in \cite[Corollary 44]{oneil}. 

We now proceed with a fundamental construction to warped products, the lifting of the Levi-Civita connection $\lf$ to $M=B\times_{f} F$.
% can now be related to those of $B$ and $F$ as follows. 
Details of the proof can be found at Proposition~35\cite{oneil}.

\begin{proposition}\label{prop:lift-conn}\cite[Proposition 35]{oneil}
On the warped manifold $M=B \times_{f} F$, if $X^{H}, Y^{H} \in \mathfrak{L}(B)$, $U^{V}, V^{V} \in \mathfrak{L}(F)$, for vector fields $X, Y$ in $B$ and $U, V$ in $F$, the Levi-Civita connections $\lB$ of $B$, $\lF$ of $F$ and $\lf$ of $M$, are related by:

\begin{itemize}
    \item [a.] $\pi_{*}(\lf_{X^{H}}Y^{H})=(\lB_{X}Y)^{H}$ y $\sigma_{*}(\lf_{X^{H}}Y^{H})=0$.
    \item[b.] $\lf_{X^{H}}U^{V}=\lf_{U^{V}}X^{H}=(\frac{X f}{f})U^{V}$. 
    \item[c.] $\pi_{*}(\lf_{U^{V}}V^{V})=-(\frac{g_f(U^{V},V^{V})}{f})grad (f)=-f g_F(U^{V},V^{V}) grad (f)$.
    \item[d.] $\sigma_*(\lf_{U^{V}}V^{V})=(\lF_{U}V)^{V}$.
    %\item[d.] $\sigma_{*}(\nabla_{U}V)=^{h}\nabla_{U} V$.
\end{itemize}   
\end{proposition}

Note that the last two items can be equivalently expressed as:

\begin{equation*}
    \lf_{U^{V}}V^{V}=-(\frac{g_f( U^{V}, V^{V})}{f})grad (f)+(\lF_{U}V)^{V}.
\end{equation*}

\begin{rmk}
    The same construction works for the liftings of any connection on $B$ and $F$. This will be useful when we consider the dual structures in next section.
\end{rmk}

Remember, a geodesic in a Riemannian manifold $M$ is a curve $\gamma: I \longrightarrow M$ whose vector field $\gamma'$ is {\it self-parallel}, i.e., $\nabla_{\gamma'}\gamma'=0$. %Equivalently, geodesics are the curves of acceleration zero: $\nabla_{\gamma'}\gamma'=0$. 
The interpretation for the geodesic curves in warped manifold is presented in the following proposition, the details of the proof can be found at \cite[Proposition 38]{oneil}.

\begin{proposition}
    A curve $\gamma(s)=(\alpha(s), \beta(s))$ in $M=B\times_{f} F$, with $\alpha$ and $\beta$ the projections of $\gamma$ into $B$ and $F$, respectively. $\gamma(s)$ is a geodesic if and only if:
    \begin{itemize}
        \item[a.] $\nabla_{\alpha'}\alpha'=(\beta',\beta')(f \circ \alpha) grad(f)$ $\in B$.
        \item[b.] $\nabla_{\beta'}\beta'=\frac{-2}{f \circ \alpha} \frac{d(f \circ \alpha)}{ds}\beta'$ $\in F$.
    \end{itemize}
\end{proposition}

In the case of a fixed point $\beta_0\in F$, the curve $\gamma(s)=(\alpha(s), \beta_{0})$ in $B \times_{f}F$ is a geodesic if and only if $\alpha(s)$ is geodesic in $B$. In an analogous way, the curve  $\gamma(s)=(\alpha_{0}, \beta(s))$ in $B \times_{f}F$ (with fixed point $\alpha_0$ in $B$) is geodesic if and only if $\beta$ is geodesic in $F$. However, it is not true that $\gamma=(\alpha(s), \beta(s))$ is geodesic if and only if $\alpha$ and $\beta$ are geodesics in $B$ and $F$, respectively.

\section{Statistical model and its structure}

The purpose of this section is to define the statistical model and its structure, based on differential manifolds whose points are probability distributions. First, we construct manifolds as finite sets of signed measurements, which statistically correspond to the sample space of any particular event. Second, we define the Fisher metric on tangent vectors of the manifold as an inner product. Considering the above, we define a statistical model and spacial classes of families to study its geometric structure. For more details, associated concepts, and application we refer to \cite{amariMeth} and \cite{jost} and reference therein.

\subsection{Finite sample space}
A random process formally consists of a probability space $(\Omega, S, p)$, where $\Omega$ is the sample space obtained of all the possible results of a randomized experiment in the data collection, $S$ is the collection of events to be studied, i.e., a subset of the sample space and $p$ is a function that returns the probability of an event.\\
%this is a $\sigma$-algebra that consisting of signed measures $\mu$ generated by a random variable $X: \Omega \longrightarrow \mathbb{R}$.

From a particular geometrical perspective \cite{jost}, we will determine the probability space for finite sample spaces, i.e., $\Omega=I$ where $I=\{1,2,..., n\}$. 
%If $f$ is a correspondence between $i\in I$ and $\mathbb{R}$, this is $f: I\longrightarrow \mathbb{R}$ the set of all these functions is 
From linear algebra, it is known that $S(I)$ is a vector space (as the space of real-valued function from $I$, and with dual space $S(I)^*$), that can be identify with the linear forms $\mu: S(I)^*\longrightarrow \mathbb{R}$, which in canonical dual basis $\{\delta^{1}, \cdots \delta^{n} \}$, is written as: 

\begin{equation}\label{mu}
\mu=\sum_{i\in I} \mu_i \delta^{i}, 
\end{equation}
where each element $\delta^{i}$ is the dual  covector, that is $\delta^i(j)=\delta_{ij}$. Indeed, the election of this basis allows us to define the identification (as smooth manifold) $\psi: S(I) \longrightarrow \mathbb{R}^{n}$  by the coordinate map:

\begin{equation*}
     \psi(\mu)=\psi(\sum_{i\in I} \mu_i \delta^{i})=(\mu_{1},\cdots, \mu_{n}).
 \end{equation*}
In other words, this is the coordinate function of $S(I)$ which endows a differentiable structure for $S(I)$ with local model $\mathbb{R}^n$. %and the surjective is a consequence of the generating basis, then $\psi$ is a coordinate chart in space.\\
%========coord for S(I)
%Let us consider another coordinate system $\phi: S(I) \longrightarrow \mathbb{R}^{n}$, associated to another basis $\{\alpha^{1}, \cdots, \alpha^{n}\}$, the change of coordinates $\psi \circ \phi^{-1}: \mathbb{R}^{n} \longrightarrow \mathbb{R}^{n}$ is a linear function. For linearity, $\psi \circ \phi^{-1}$ y $\phi \circ \psi^{-1}$ are differentiable, thus $S(I)$ is a \textbf{differentiable manifold}. 
%For further purposes, it will be helpful to introduce some special submanifolds of $S(I)$. Let us consider the function $\phi: S(I) \longrightarrow \mathbb{R}$, defined by, $\phi(\mu)=\sum_{i=1}^{n}\mu_{i}$ to study the classical situation of level set manifolds, i.e. $\phi^{-1}(a)$ as a submanifold of $S(I)$
%
%\begin{equation*}
%    S_{a}(I):=\{ \sum_{i\in I} \mu_i \delta^i: \phi(\mu)=\sum_{i\in I} \mu_i =a, a\in \mathbb{R} \}.
%\end{equation*}

A first goal is to  endow with a geometry the space of probability. Let us begin with the open submanifold $\mathcal{M}(I)$ of $ S(I)$, as the {\it positive measures} on $I$:
\begin{equation*}
    \mathcal{M}_+(I):=\{\mu \in S(I):\mu_i > 0, \forall i\in I \},
\end{equation*}
 whose closure is given by the {\it non-negative measures}:
\begin{equation*}
    \mathcal{M}(I):=\{\mu \in S(I):\mu_i \geq 0, \forall i\in I \},
\end{equation*}
as {\it manifold with corners}. From a regular value of the map  $\phi: S(I) \longrightarrow \mathbb{R}$ (defined by $\phi(\mu)=\sum_{i=1}^{n}\mu_{i}$), we get: 
\begin{equation}\label{simplex}
\mathcal{P}_{+}(I):=\mathcal{M}_{+}(I) \cap \phi^{-1}(1)=\{\mu \in \mathcal{P}(I):\mu_i > 0, \forall i\in I, \sum_{i\in I}\mu_i=1 \},   
\end{equation}
as level set on the open submanifold $\mathcal{M}_{+}(I)$. And also, we have the closure of $\mathcal{P}_{+}(I)$, as:
\begin{equation*}
     \mathcal{P}(I):=\{\mu \in \mathcal{M}(I):\mu_i \geq 0, \forall i\in I, \sum_{i\in I}\mu_i=1 \}.
\end{equation*}

It worth to mention that, as $\phi^{-1}(1)$ is a level set, then $\mathcal{P}_{+}(I)$ decrease in 1 the dimension of $\mathcal{M}_{+}(I)$.
%Note that submanifolds such as $\mathcal{P}(I)$ y $\mathcal{P}_{+}(I)$ are formed by probability density points. The first approach to these submanifolds, is to find the coordinate system for $\mathcal{P}_{+}(I)$ called in statistics \textbf{the probability simplex}.

\begin{example}\label{ejem1}
For the set of natural numbers $I=\{1, \cdots , n, n+1\}$ let us consider: 
\begin{equation*}
    \textbf{U}:=\{x=(x_1, \cdots, x_n)\in \mathbb{R}^n: x_i>0,  \forall i \in I,  \sum_{i=1}^{n}x_i< 1\},
\end{equation*}
and the application $\varphi(x):\textbf{U} \longrightarrow \mathcal{P}_{+}(I),$ by:
%where $x\in \textbf{U}$ is $n$-dimensional and $\mathcal{P}_{+}(I)$ has $n$ coordinates. Thus, the coordinate system for probability simplex $\mathcal{P}_{+}(I)$ is:
\begin{equation}\label{coor}
   \varphi(x)=\sum_{i=1}^{n}x_i\delta^{i}+(1-\sum_{i=1}^{n}x_i)\delta^{n+1},
\end{equation}
that defines the smooth structure for $\mathcal{P}_{+}(I)$, and gives us explicit coordinates system.
$\diamond$
\end{example}

\subsection{Fisher metric and statistical model}

The goal of this section is to introduce an inner product that considers the pointwise data in $\mathcal{M}_{+}(I)$. This idea will be promoted to a metric on the whole manifold $\mathcal{M}_{+}(I)$. For this, we will begin with an inner product that will depend on each element $\mu$ of the set of positive measures on $I$:
%$$f\longmapsto f \mu:=\langle f, \cdot \rangle_{\mu}$$
\begin{definition}[\cite{jost}]\label{productointerno}
  Let us fix $\mu\in \mathcal{M_{+}}(I)$ and define the inner product $\langle \cdot,\cdot \rangle_{\mu}$ in $S(I)$ as follows:  
\begin{equation}\label{eq4}
    \langle f, g \rangle_{\mu}=\mu \cdot (fg)=\sum_{i=1}^n\mu_if_ig_i
\end{equation}
for any $f,g\in S(I)^*$.
\end{definition}

It is well known that $S(I)$ and $S(I)^*$ are canonical isomorphic, but we also can obtain a family of isomorphisms parametrized by $\mathcal{M_{+}}(I)$. Just note that when consider the basis  $\{e_i\}$ and  $\{\delta^i\}$ on $S(I)^*$ and $S(I)$ respectively, the element $\mu \in \mathcal{M}_{+}(I)$ (with representation $\mu=\sum_{i\in I}\mu_{i}\delta^{i}$) induces an isomorphism between $S(I) $and $ S(I)^{*}$ by  $\frac{da}{d\mu}:=\sum_{i\in I}\frac{a_{i}}{\mu_{i}}e_{i}$ in $S(I)^*$. Thus, in the definition \ref{productointerno} we rewrite the relation \eqref{eq4} in $S(I)$ by:
\begin{equation}\label{metr}
    \langle a, b \rangle_{\mu}=\langle \dfrac{da}{d\mu}, \dfrac{db}{d\mu} \rangle_{\mu}=\sum_{i}\frac{1}{\mu_{i}}a_{i}b_{i}
\end{equation}

with $a,b\in S(I)$.\\      

In order to promote the inner product ~\eqref{metr} to a Riemannian metric on $\mathcal{M}_{+}(I)$ we first define its tangent space. It is a straightforward computation to verify that tangent space of a vector space (seen as a manifold) is again the same vector space. In our scenario, we get the following:
%The tangent space of a manifold is studied locally, for finite dimensional vector space  $S(I)$ is:
\begin{equation}\label{TS}
    T_{\mu}S(I) \cong \{\mu\} \times S(I), \quad \mbox{\ and\ } \quad T_{\mu}\mathcal{M}_{+}(I) \cong \{\mu\} \times S(I).
\end{equation}
%as an open submanifold of $S(I)$, $\mathcal{M}_{+}(I)$ has the same tangent space at the point $\mu \in \mathcal{M}(I)$, thus: $T_{\mu}\mathcal{M}_{+}(I) \cong \{\mu\} \times S(I)$.\\
%\begin{figure}[h]
 %   \centering    \includegraphics[width=0.5\textwidth]{metricFisher-m+.jpg.png}
%\caption{Fisher metric on $\mathcal{M}_{+}(I)$ with  signed measure $\mu=(4,5,2)$, vectors $A=(4.5,5.5,2.5)$ and $B=(2.7,3.5,3.5)$ is:
%$$\textbf{g}_{\mu}(A, B)=3.35$$}
%\end{figure}

\begin{definition}\cite[Definition 2.1]{jost}\label{defmodeloesta}
 \textbf{The Fisher metric} (or metric tensor $\mathbf{g}$) on $\mathcal{M}_{+}(I)$ is defined on each $\mu \in \mathcal{M}_{+}(I)$ by $\mathbf{g}_{\mu}: T_{\mu}\mathcal{M}_{+}(I) \times T_{\mu}\mathcal{M}_{+}(I) \longrightarrow \mathbb{R} $ such that, for two tangent vectors $A\sim (\mu,a),B\sim (\mu,b)\in T_\mu \mathcal{M}_{+}(I)$
\begin{equation}\label{eq7}
   \mathbf{g}_{\mu}(A, B):= \langle a, b \rangle_{\mu}.
\end{equation}

A \textbf{statistical model} for a $n$-dimensional manifold $M$ is a pair $(g,p)$, with $g$ a Riemanninan metric in $M$ and an embedding $p:M \hookrightarrow \mathcal{M}_{+}(I)$ ($\xi\in M\mapsto {p}(\xi)=\sum_{i\in I} p_{i}(\xi)\delta^{i}$), such that the pull-back of the Fisher metric coincides with $g$,
%$\textbf{g}$ define una m\'{e}trica en $M$, 
i.e. for $X, Y \in T_{\xi}M$    
\begin{equation}\label{metrifisherpullback}
  g_{\xi}(X,Y)=\mathbf{g}_{{p}(\xi)}(dp_{\xi}X,dp_{\xi}Y).
\end{equation}
\end{definition}

The main example of statistical model is the space $M=\mathcal{P}_+(I)$ with the natural embedding on $\mathcal{M}_+(I)$.

In particular, if this manifold is composed of points of probability, it gives rise in statistics to Fisher information matrix.\footnote{Indicates how much information about an unknown parameter we can obtain in a sample space.} To illustrate the above, we take the manifold $\mathcal{P}_{+}(I)$ of example \ref{ejem1}, obtaining the respective information matrix associated with Fisher metric:

\begin{example}\label{ej:F.mezcla}
    
Consider the coordinate system~\eqref{coor} of $\mathcal{P}_{+}(I)$. The coefficients of the information matrix associated with Fisher's metric are given by:
\begin{equation}\label{Coef-fisher-simplejo}
    g_{ij}(\mu)=\sum_{k=1}^{n}\frac{1}{\mu_k}\delta_{ki}\delta_{kj}+\frac{1}{\mu_{n+1}}=
\begin{cases}
\frac{1}{\mu_i}+\frac{1}{\mu_{n+1}}
& \mbox {si i=j}\\
\frac{1}{\mu_{n+1}} & \mbox{en otro caso}
\end{cases}    
\end{equation}
Explicitly we have:

\begin{equation*}
    G(\mu):=(g_{ij}(\mu))=\frac{1}{\mu_{n+1}} \begin{pmatrix}
        \frac{\mu_{n+1}}{\mu_{1}}+1  &  1 & \dots & 1\\
        1 &  \frac{\mu_{n+1}}{\mu_{2}}+1 &\dots & 1\\
        \vdots & \vdots & \ddots & \vdots\\
        1 & 1 & \dots & \frac{\mu_{n+1}}{\mu_{n}}+1\\
    \end{pmatrix}
\end{equation*}
The inverse matrix of $G(\mu)$ is the probability covariance matrix of $\mu$, each element $i\in{1, \dots n}$ has a probability $\mu_{i}$, the coefficients are:
$$g^{ij}(\mu)=\begin{cases}
\mu_{i}(1-\mu_{i})
& \mbox {si i=j}\\
-\mu_{i}\mu_{j} & \mbox{en otro caso}
\end{cases}$$

In statistics, the diagonal of the matrix $G^{-1}(\mu)$, are the values of the variances of the variables and the remaining coefficients give the value of the correlation between the variables.
In fact, this is the statistical origin of the Fisher metric as a covariance matrix \cite{rao}.$\diamond$

%\begin{figure}[h]
    %\centering   \includegraphics[width=0.7\textwidth]{metricFisher-P+.png}
    %\caption{Fisher metric on $\mathcal{P}_{+}(I)$ with signed measure $\mu=(0.01,0.02,0.98)$ and $\lambda=(0.\bar{3}, 0.\bar{3}, 0.\bar{3})$, vectors $A=(0.01, 0.03, 0.04)$ and $B=(0.01,0.15, 0.02)$ thus:\\
      %  $$\textbf{g}_{\mu}(A,B)=0.2358  \quad\textbf{g}_{\lambda}(A,B)=0.0162$$}
    %\label{fig:metrcP+}
%\end{figure}

%Note, the Fisher metric in $\mu$ has more information than in $\lambda$, with 23.58\%.
\end{example}

\begin{example}
In the case of $I=\{1,2\}$, we choose two points $\nu=(1/3,1/3,1/3)$ and $\mu=(0.12,0.08,0.80)$ for which the Figure \ref{fig:p+}. shows the centered balls (for the Fisher metric) with same radio in $\mathcal{P}_+(I)$.

%{\color{blue} Note that points in the center concentrate the same information, in comparison with points closer to the {\it boundary} these accumulate more information in one of the coordinates.}

%include some graphic to expalin this as in talk in the thesis...think is better this picture
\begin{figure}[h]
    \centering
    \includegraphics[scale=1]{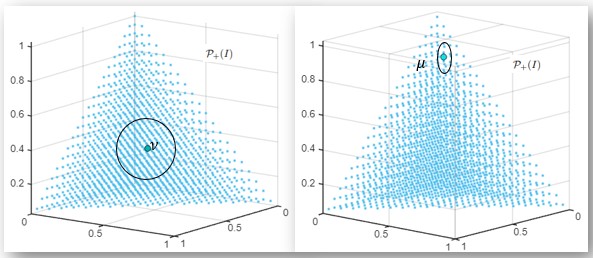}
    \caption{Balls in Fisher metric centered at $\nu$ and $\mu$ with same radii}
    \label{fig:p+}
\end{figure}
 It worth to note that the meaning of the size of the ball is the quantity of information at each point. To clarify this claim, note that the third coordinate in $\mu$ is closed to 1, the other two coordinates must be closed to 0, but for $\nu$ this data is homogeneous. For this reason the ball centered at $\mu$  will have more restriction that the one centered at $\nu$, i.e, we get more information on $\mu$.
$\diamond$
\end{example}

%{\color{red}not sure if we include this:
%Given $\mathcal{P}_{+}(I)$ a submanifold whose points are probability and additional,
% $\mathcal{P}_{+}(I)\subseteq \mathcal{M}_{+}(I) \subseteq S(I)$. Considering an arbitrary $n-$dimensional manifold $M$ and the function including $\textbf{p}:M \hookrightarrow \mathcal{M}_{+}(I)$ as $\xi \mapsto \textbf{p}(\xi)=\sum_{i\in I} p_{i}(\xi)\delta^{i}$ an embedding\footnote{An immersion that is injective and proper is an embedding.} to result in the following definition:}

In order to find the expression of $g$, it is sufficient to take the canonical directions of tangent $X=\frac{\partial p_i}{\partial \xi_l}, Y=\frac{\partial p_i}{\partial \xi_j}$ in $\mathcal{M}_{+}(I)$, to obtain:
 
$${\mathbf{g}}_{\xi}(X,Y) =\sum_{i\in I} \frac{1}{p_{i}(\xi)}\frac{\partial p_{i}}{\partial \xi_{l}}(\xi) \frac{\partial p_{i}}{\partial \xi_{j}}(\xi).$$

Now, applying the logarithmic derivative the Fisher metric is rewritten as: 

$${\mathbf{g}}_{\xi}(X,Y) =\sum_{i \in I} p_{i}(\xi) \frac{\partial \log p_{i}}{\partial \xi_{l}}(\xi) \frac{\partial \log p_{i}}{\partial \xi_{j}}(\xi).$$

This representation of the Fisher metric is more familiar in the standard treatment of information geometry, using the definition of expected value: %or mathematical expectation for a discrete sample, $E(x)=\sum_{i} x_{i} \cdot p(x_{i})$ we have
\begin{equation}\label{metricafisheresperanza}
    {\mathbf{g}}_{\xi}(X,Y):=E[\frac{\partial \log p_{i}}{\partial \xi_{l}}(\xi) \frac{\partial \log p_{i}}{\partial \xi_{j}}(\xi)]
\end{equation}

Using again the notion of logarithmic derivative but on the statistical model  $\mathcal{P}_+(I)$, we obtain:
$$\sum_{i \in I} p_i\frac{\partial \log p_{i}}{\partial \xi_{l}}=\sum_{i \in I} \frac{\partial p_{i}}{\partial \xi_{l}}=\frac{\partial }{\partial \xi_{l}}\sum_{i \in I} p_i=0,$$ therefore we obtain:
\begin{align*}
0=\frac{\partial }{\partial \xi_{l}} E[\frac{\partial }{\partial \xi_{j}}\log p_i]&= \sum_{i} \frac{\partial }{\partial \xi_{l}}(p_i\frac{\partial }{\partial \xi_{j}}\log p_i)= \sum_{i} ((\frac{\partial }{\partial \xi_{l}}p_i)\frac{\partial }{\partial \xi_{j}}\log p_i)+(p_i\frac{\partial }{\partial \xi_{l}}\frac{\partial }{\partial \xi_{j}}\log p_i)\\
&=\sum_{i} (p_i(\frac{\partial \log p_i}{\partial \xi_{l}})\frac{\partial \log p_i}{\partial \xi_{j}})+(p_i\frac{\partial }{\partial \xi_{l}}\frac{\partial }{\partial \xi_{j}}\log p_i)={{g}}_{\xi}(X,Y)+\sum (p_i\frac{\partial }{\partial \xi_{l}}\frac{\partial }{\partial \xi_{j}}\log p_i),
\end{align*}

which finally yields another equivalent way to define the Fisher metric but on $\mathcal{P}_+(I)$:

\begin{equation}\label{metricafisheresperanza-1}
    g_{\xi}(X,Y):=-E[\frac{\partial^2 \log p_{i}}{\partial \xi_{l} \partial \xi_{j}}(\xi) ].
\end{equation}

%======not sure if put here========\\
%Similar to the above. We now present, the symmetric 3-tensor has representation with respect to the expected value, with the following result:

%{\color{red}havent mention what is tensor AC}
%\begin{theorem}\label{unicidad-gyT}[Theorem 1.2 en \cite{jost}]
%The Fisher metric tensor and the Amari-Chentsov 3- tensor given by:
%    $$ T_{ijk} :=E_{\textbf{p}}[\frac{\partial}{\partial \xi^{i}}\log p \frac{\partial}{\partial \xi^{j}}\log p \frac{\partial}{\partial \xi^{k}}\log p]$$
%    are the only (up to a constant scaling factor) that are invariant under sufficient statistics.
%\end{theorem}
%=============\\

\begin{rmk}\label{rmk:fin vs infn}
    
As is usual in probability theory, under suitable considerations, everything that has been done so far with finite $I$ can be extended in the same way to measurable spaces with probability measures, where the set $I$ is not necessarily finite, but is endowed with a $\sigma$-algebra and is a measurable space (for more details we refer to  \cite{jost} Section 3.2.1). In the not-finite case, we should obtain again $\mathcal{M}_+(I)$ and $\mathcal{P}_+(I)$ finite dimensional manifolds. This will be done by describing the variable $x$ (in the measurable space) with some finite number of parameters that will play the role of coordinate system. Hereafter, we will have this consideration for the Fisher metric (and associated geometric terms). Particular examples are {\it exponential families} and also {\it mixture families} studied in the following section.

\end{rmk}
% In particular, we will endow the half plane $\mathbb{R}\times\mathbb{R}^+$ with a Fisher metric that codifies the one-dimensional the normal distributions, that is, each point $\mathbb{R}\times\mathbb{R}^+$ has a normal distribution.  

In the following example, we consider a space of distributions parameterized by two terms, where we change the finite sum by integration in the definition of the Fisher metric:

\begin{example}\label{ej:1-normal}
Consider the normal distribution %$\mathcal{N}(x;\mu,\sigma)$ 
$x$ as real random variable with  $\xi=(\mu, \sigma)$ where the parameters are the mean $\mu$  and the standard deviation $\sigma$, that is:% and $\int_{\mathbb{R}}\mathcal{N}(x;\mu, \sigma)dx=1$:
\begin{equation}\label{Normal}
    \mathcal{N}(x;\mu, \sigma)=\frac{1}{\sqrt{2\pi \sigma^2}}\exp (-\frac{(x-\mu)^2}{2\sigma^2}).
\end{equation}
In order to apply the formulae for the Fisher metric \eqref{metricafisheresperanza-1} in this case with, we should note that  
$$\ln {\mathcal{N}(x;\mu, \sigma)}=-\frac{(x-\mu)^2}{2\sigma^2}-\frac{1}{2}\ln(2\pi \sigma^2).$$ 
When we take partial derivatives and integration with respect to the distribution (in the lines as Remark~\ref{rmk:fin vs infn}), we obtain

%Therefore, the matrix associated with the Fisher metric $G(\mu, \sigma)=(g_{ij})$ es:

\begin{equation}\label{metricanormal-mu-sigma}
    G(\mu, \sigma)=\begin{pmatrix}
\frac{1}{\sigma^{2}} & 0\\
0 & \frac{2}{\sigma^{2}}
\end{pmatrix}.
\end{equation}
For a detailed computation for this metric, we  reefer to \cite[section 1.3]{olga}. $\diamond$
\end{example}

Note that, the previous example says that the statistical model for normal distribution identifies {\it isometrically} with the hyperbolic space:
$$(H:=\{(\mu, \sigma): \mu\in \mathbb{R}, \sigma > 0 \},\frac{d\mu^{2}+2d\sigma^{2}}{\sigma^{2}}).$$
A remarkable geometrical consequence of this identification lies in the distance-minimizing curves in this geometry. Although, in the usual euclidean space the minimizing distance between two normal distributions is a segment of a line in the plane, the Fisher metric is not. Indeed, a geodesic is a section of a semicircle  because these are the geodesic in the hyperbolic plane\footnote{which gives the minimal distance between two normal distributions.},
as shown in Figure \ref{fig:enter-label}:

%Fisher's metric, as in ~\eqref{metricanormal-mu-sigma}, that is, $\frac{d\mu^{2}+2d\sigma^{2}}{\sigma^{2}}$ coincides in the direction of $\mu\in\mathbb{R}$, i.e,  up to a simple scaling is the hyperbolic metric of the half-plane $\frac{d\mu^{2}+d\sigma^{2}}{\sigma^{2}}$. 

\begin{figure}[h]
    \centering
\includegraphics[width=0.6\textwidth]{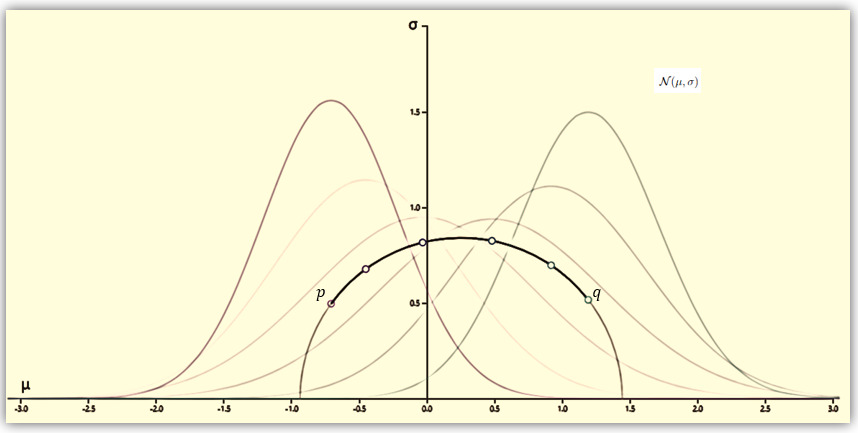}
    \caption{ %Equidistant pairs in Fisher metric. Geodesic, 
    Shortest path between the normal distributions $p=\mathcal{N}(-0.75,0.5)$ and $q=\mathcal{N}(1.25,0.5)$ in the Fisher metric \cite{Rit} .}
    \label{fig:enter-label}
\end{figure}

%=======this not======\\
%One of the goals of this manuscript, is to give an overview of the global geometry of information geometry. For this reason, we give another approach to the Fisher metric via {\it divergences}
%\begin{align}\label{divCanonica}
 %    D(p\|q)&:=\int_{0}^{1} t\| \dot{\gamma}_{p,q}(t)\|^{2}dt,
%\end{align}

%given $q,p\in M$, there is only one curve geodesic
%$$\gamma_{q,p}:[0,1]\longrightarrow M,$$ 
%which satisfies $\gamma_{q,p}(0)=p$ y $\gamma_{q,p}(1)=q$.\\
%==========\\

%The statistics models are classified according to the algebraic form of their probability function. 

\subsection{Exponential and mixture families in geometry}

%In statistics, the study of probability functions corresponding to variables associated with natural phenomena (e.g., temperature, height, rainfall) involves understanding the distributions of these variables in terms of parameters like mean, variance, and other moments, that describe the central tendency, spread, and other characteristics of the distributions. 
To take advantage of the notions and concepts presented so far, we will use particular forms of probability distributions that should be considered. From this point of view, we will consider special families that simplify the way we describe such distribution. In the literature, we find two such families, known as exponential and mixture families. Geometrically, these families will specify coordinate systems and connections, which will determine aspects such as geodesics or special directions for derivatives.\\

\textbf{The exponential family} is a broad class of probability distributions that share a common mathematical form. Many well-known distributions, such as the normal, Poisson, binomial, and exponential, belong to this family. A key characteristic of this family is that the distributions can be expressed in a standard form, which yields computational efficiency in classical inference problems. The associated \textit{coordinate system}  to this model is denoted by $(e)$. Formally, it consists on probability distributions of the form:

\begin{definition}\cite[Definition 3.1]{jost}
    An \textbf{exponential family} is a family of probability distributions  $p(\cdot; \vartheta)$ with embedding  ${p} :M \hookrightarrow \mathcal{P}_{+}(\Omega)$, of the form:
    \begin{equation}\label{famiExp}
        p(x; \vartheta)=\exp[\gamma(x)+\sum_{i=1}^{n}f_{i}(x)\vartheta^{i}-\psi(\vartheta)],
        \end{equation}
where, $x$ is a real random variable, $\vartheta=(\vartheta^{1}, \dots, \vartheta^{n})$ is a $n$-dimensional parameter with function $\gamma(x)$, $f_{1}(x), \dots f_{n}(x)$ and  $\psi(\vartheta)$, under the normalized condition $\int_M p(x; \vartheta)dx=1$.   
%y $\mu(x)$ es una medida sobre $\Omega$.
\end{definition}

From the defining relation \eqref{famiExp} of exponential family, we have $    \log p(x; \vartheta)= \gamma(x)+\sum_{i=1}^{n}f_{i}(x)\vartheta^{i}-\psi(\vartheta)$, which gives us: 
$$ \frac{\partial^2 \log p(x; \vartheta)}{\partial \vartheta^{i}\partial \vartheta^{j}}= -\frac{\partial^2 }{\partial\vartheta^{i} \partial \vartheta^{j}}\psi(\vartheta).$$ As the right hand side is independent of the variable $x$, we get that: 

$$-\int_M \frac{\partial^2 \log p(x; \vartheta)}{\partial \vartheta^{i}\partial \vartheta^{j}} p(x;\vartheta) dx = \frac{\partial^2 }{\partial\vartheta^{i} \partial \vartheta^{j}}\psi(\vartheta)\int_M  p(x;\vartheta) dx,$$

which finally gives us an equivalent version of \eqref{metricafisheresperanza-1} as:

\begin{align}\label{componentesmatrizFisher}
    g_{ij}(p)=-E\Big[ \frac{\partial^2 \log p(x; \vartheta)}{\partial \vartheta^{i}\partial \vartheta^{j}} \Big] &= \frac{\partial^2 }{\partial\vartheta^{i} \partial \vartheta^{j}}\psi(\vartheta),
\end{align}

thus, the function $\psi(\vartheta)=\log \int \exp[\gamma(x)+\sum_{i=1}^{n}f_{i}(x)\vartheta^{i}]dx$  allows us to give an explicit  expression for the components of the information matrix associated to the Fisher metric. Therefore, the exponential family is a statistical model whose Fisher  metric is defined by the coefficients in~\eqref{componentesmatrizFisher}. The example \ref{ej:1-normal} can be revisited in terms of exponential family as  follows:
 
\begin{example}
Rewriting the expression \eqref{Normal}, we get:

\begin{equation*}
    \mathcal{N}(x;\mu, \sigma^2)=\exp(\ln(\frac{1}{\sqrt{2\pi \sigma^2}}))\exp(-\frac{x^2-2\mu x+\mu^2}{2\sigma^2})=\exp(\frac{\mu x}{\sigma^2}-\frac{x^2}{2\sigma^2}-\frac{\mu^2}{2\sigma^2}-\ln(\sqrt{2\pi \sigma^2})),
\end{equation*}
and use a new inner product to obtain
\begin{equation}\label{expNormal}
    \mathcal{N}(x;\mu, \sigma^2)=\exp((x,x^2)\cdot(\frac{\mu}{\sigma^2},-\frac{1}{2\sigma^2})-\frac{\mu^2}{2\sigma^2}-\frac{1}{2}\ln(2\pi\sigma^2)),
\end{equation}
expressed in new variables with the following change of variables $\vartheta^{\textbf{1}}=\frac{\mu}{\sigma^2}$ and $\vartheta^{\textbf{2}}=-\frac{1}{2\sigma^2}$. The normal distribution takes the form \eqref{famiExp} with parameters $\vartheta=(\vartheta^{\textbf{1}}, \vartheta^{\textbf{2}})$ and we note that $\frac{\mu^2}{2\sigma^2}=-\frac{(\vartheta^{\textbf{1}})^2}{4\vartheta^{\textbf{2}}}$ y $\sigma^{2}=-\frac{1}{2\vartheta^{\textbf{2}}}$, which directly implies that:

\begin{equation*}
    \mathcal{N}(x;\vartheta^{\textbf{1}}, \vartheta^{\textbf{2}})=\exp\Big((x,x^2)\cdot(\vartheta^{\textbf{1}}, \vartheta^{\textbf{2}})-\Big(-\frac{(\vartheta^{\textbf{1}})^2}{4\vartheta^{\textbf{2}}}+\frac{1}{2}\ln(\pi)-\frac{1}{2}\ln(-\vartheta^{\textbf{2}})\Big)\Big),
\end{equation*}
where the function $\psi(\vartheta)$ is:
\begin{equation*}
     \psi(\vartheta)=-\frac{(\vartheta^{\textbf{1}})^2}{4\vartheta^{\textbf{2}}}+\frac{1}{2}\ln(\pi)-\frac{1}{2}\ln(-\vartheta^{\textbf{2}}).
\end{equation*}
According to ~\eqref{componentesmatrizFisher}, the components of the Fisher information matrix for the exponential family are calculated with the second partial derivatives in each parameter of $\psi(\vartheta)$ yielding again~\eqref{metricanormal-mu-sigma}.
    $\diamond$
\end{example}
%In particular, the exponential family a {\textit{coordinate system}} in $M$, denoted as $(e)$.\\ 

The mixture family is a class of statistical models used when the observed data is thought to be generated by multiple underlying processes, each with its probability distribution. These models are useful for clustering or classifying data into subgroups when the population is not homogeneous.  The associated coordinate system to this model is denoted as $(m)$. Formally, it consists of probability distributions of the form:
%which corresponds to \eqref{coor}.

\begin{definition}\cite[section 4.2]{jost}
    An \textbf{mixture family} is a  family of probability distributions $p(\cdot; \eta)$ with embedding ${p} :M \hookrightarrow \mathcal{P}_{+}(\Omega)$, of the form:
\begin{equation}\label{FamiliaMezcla}
    p(x;\eta)=c(x)+\sum_{i=1}^{d}h^{i}(x)\eta_{i},
\end{equation}
where $x$  is a real random variable, $\eta=\eta_{1}, \dots, \eta_{d}$ is a parameter $d$-dimensional, $c(x)$ y $h^{1}(x), \dots, h^{d}(x)$ are integrable functions  under normalisation property $\int_M c(x) dx=1$ and $\int_M h^{i}(x)dx=0$
\end{definition}

From the expression \eqref{FamiliaMezcla} we have $\frac{\partial^2}{\partial \eta_{i}\partial \eta_{j}} \log p(x; \eta)=-\frac{h^{i}(x)h^{j}(x)}{[p(x;\eta)]^2}$, which yields (cf. \cite[expression (94)]{frank2020}):
\begin{equation}
  {g_{ij}(p)}=\int_M\frac{h^{i}(x) h^{j}(x)}{p(x; \eta)}dx,
\end{equation}

which corresponds to the coefficients of the matrix associated with the Fisher metric. %Therefore {\textit{the mixture family is a statistical model}}.

An example of probability distribution grouped in this family, is the collection of probability functions for a finite sample, as was presented in the example \ref{ejem1} and coefficients in \eqref{Coef-fisher-simplejo}.

Once we fix these two families, we proceed to  obtain more geometrical data. One of the possible geometrical notion we can apply is the {\it parallel transport}.

\begin{definition}\cite[section 2.4]{jost}
Let  $\mu$ and $\nu$ be two points in $\mathcal{M}_{+}(I)$, and  $A \sim (\mu, a)$ in $T_{\mu}\mathcal{M}_{+}(I)$ a tangent vector. Denote  ${\Pi}_{\mu,\nu}^{(e)}$ as a parallel transport in $T\mathcal{M}_{+}(I)$ with coordinate system $(e)$, given by:

\begin{align}  \nonumber
{\Pi}_{\mu,\nu}^{(e)}:T_{\mu}\mathcal{M}_{+}(I)&\longrightarrow T_{\nu}\mathcal{M}_{+}(I)\\   \label{defTrans(e)}
a &\mapsto (\nu, (\tilde{\phi}_{\nu}^{-1}\circ \tilde{\phi}_{\mu})(a))=\sum_{i}\nu_{i}\frac{a_i}{\mu_{i}}\delta^{i}. 
\end{align}

Denote  ${\Pi}_{\mu,\nu}^{(m)}$ as a parallel transport in $T\mathcal{M}_{+}(I)$ with coordinate system $(m)$, given by:
\begin{align}\nonumber {\Pi}_{\mu,\nu}^{(m)}:T_{\mu}\mathcal{M}_{+}(I) &\longrightarrow T_{\nu}\mathcal{M}_{+}(I)\\  \label{defTrans(m)}
    a & \longmapsto \sum_{i}a_{i}\delta^{i}=a.
\end{align}
\end{definition}

%Similarly, in $\mathcal{M}_{+}(I)$ for the coordinate system $(m)$ of the mixture family, we have:

%\begin{definition}
 %Given two points $\mu$ y $\nu$ in $\mathcal{M}_{+}(I)$, the tangent vector $A_{\mu}\sim (\mu, a)$ in $T_{\mu}\mathcal{M}_{+}(I)$.  
 %Let ${\Pi}_{\mu,\nu}^{(m)}$ interpret as a parallel transport in $T\mathcal{M}_{+}(I)$ with coordinate system $(m)$, by:
%\end{definition}

The remarkable behavior of metric concerning parallel transport in $A \sim (\mu, a)$ and $B \sim (\nu, b)$ is: 
\begin{equation}\label{ProFisheTrans}
       {g}_{\nu}({\Pi}_{\mu,\nu}^{(e)}A, {\Pi}_{\mu,\nu}^{(m)}B)=\sum_{i}\frac{1}{\nu_{i}}(\nu_{i}\frac{a_{i}}{\mu_{i}})b_{i}=\sum_{i}\frac{1}{\mu_{i}}a_{i}b_{i}={g}_{\mu}(A,B), 
\end{equation}

indicating the invariance of the Fisher metric, for the two parallel transports.\\

\begin{rmk}
  An {\bf affine connection} is a $\mathbb{R}$-linear map $\cc:\Gamma(TM)\times \Gamma(TM)\to \Gamma(TM)$ so that for all smooth function $f$ and any pair of vector field $X,Y$ it holds:
  \begin{equation}\label{def:conn}
\nabla_{fX}Y=f\nabla_XY \mbox{\ and\ } \nabla_XfY=(Xf)Y+f\nabla_XY.      
  \end{equation}
Affine connections are the natural extension of directional derivative when we change the metric on the configuration space.   
\end{rmk}

It is a well known fact in differential geometry, that any parallel transport defines an {\it affine connection} which is the global notion associated to derivations, torsion, curvature, and geodesics. Therefore, the parallel transports ${\Pi}_{\mu,\nu}^{(e)}$ and ${\Pi}_{\mu,\nu}^{(m)}$ will determine two types of affine connections, \cite[Proposition 2.4]{jost} as follows: given a curve $\gamma: (-\epsilon,\epsilon) \longrightarrow \mathcal{M}_{+}(I)$ with $\gamma(0)=\mu$ and $\dot{\gamma}(0)=A $, we define \textbf{affine $e$-connection} $\widetilde{\nabla}_{A}^{(e)}B|_{\mu}$ in $\mathcal{M}_{+}(I)$ as:
$$\widetilde{\nabla}_{A}^{(e)}B|_{\mu}=\lim_{t \to 0}\frac{1}{t}({\Pi}_{\gamma(t),\mu}^{(e)}(B_{\gamma(t)})-B)\in T_{\mu}\mathcal{M}_{+}(I),$$

applying \eqref{defTrans(e)} on $B=\sum_{i}b_{\mu, i}\delta^{i}$, we rewrite:
\begin{align*}
    \widetilde{\nabla}_{A}^{(e)}B|_{\mu}&=  (\mu, \lim_{t \to 0}\frac{1}{t}(\sum_{i}\mu_{i}\frac{b_{\gamma(t),i}}{\gamma_{i}(t)}\delta^{i}-\sum_{i} b_{\mu, i}\delta^{i}))\\ & =(\mu, \sum_{i} \mu_{i}\{ \frac{\frac{\partial b_{i}}{\partial a_{\mu}}(\mu) \gamma_{i}(t)- b_{\gamma(t),i}\dot{\gamma}_{i}(t)}{\gamma^{2}(t)} \}_{t=0}\delta^{i}),
\end{align*}
evaluating at $t=0$, with $\gamma_{i}(0)=\mu_{i}$ y $\dot{\gamma}_{i}(0)=a_{\mu,i}$: 
\begin{align*}
      \widetilde{\nabla}_{A}^{(e)}B|_{\mu}= & (\mu, \sum_{i} \mu_{i}\{ \frac{\frac{\partial b_{i}}{\partial a_{\mu}}(\mu) \gamma_{i}(0)- b_{\gamma(0),i}\dot{\gamma}_{i}(0)}{\gamma^{2}(0)} \}\delta^{i}) \\ =&(\mu, \sum_{i}\frac{\partial b_{i}}{\partial a_{\mu}}(\mu) \delta^{i}-\sum_{i} \frac{1}{\mu_{i}}b_{\mu,i}a_{\mu,i}\delta^{i} ),
\end{align*}
by expressions~\eqref{mu} y~\eqref{eq7}, we have:
\begin{align}\label{econexion}
    \widetilde{\nabla}_{A}^{(e)}B|_{\mu} = & (\mu, \frac{\partial b}{\partial a_{\mu}}(\mu)-\mu(\frac{d a_{\mu}}{d \mu} \cdotp \frac{d b_{\mu}}{d \mu}))=(\mu, \frac{\partial b}{\partial a_{\mu}}(\mu)-{g}_{\mu}(A, B)),
\end{align}

similarly for a curve $\gamma:(-\epsilon, \epsilon)\longrightarrow \mathcal{M}_{+}(I)$, $\gamma(0)=\mu$ and $\dot{\gamma}(0)=A$, we define \textbf{affine $m$-connection}
${\widetilde{\nabla}}_{A}^{(m)}B|_{\mu}$ as: $$\widetilde{\nabla}_{A}^{(m)}B|_{\mu}=\lim_{t \to 0}\frac{1}{t}({\Pi}_{\gamma(t),\mu}^{(m)}(B_{\gamma(t)})-B)\in T_{\mu}\mathcal{M}_{+}(I),$$
also, by \eqref{defTrans(m)} in the previous relation, we get: 

\begin{equation}\label{mconexion}
     \widetilde{\nabla}_{A}^{(m)}B|_{\mu} = (\mu,  \lim_{t \to 0} \frac{1}{t} (\sum_{i}b_{\gamma(t),i}\delta^{i}-\sum_{i}b_{\mu,i}\delta^{i}))
     =(\mu, \frac{\partial b}{\partial a_{\mu}}(\mu)),
\end{equation}
as we have that the main model for information geometry is the space of probability measures $\mathcal{P}_+(I)$ with the induced Fisher metric, we must consider connection on $\mathcal{P}_{+}(I)$. For this, we need to project the connections on the tangent space of probability measure as follows: 
\begin{align}\label{eq:m-p+conn}
\nabla^{(m)}_AB&=\widetilde{\nabla}^{(m)}_AB,  \\ \label{eq:e-p+conn}
  \nabla^{(e)}_AB &= (\mu, \sum_{i}\frac{\partial b_{i}}{\partial a_{\mu}}(\mu) \delta^{i}-\sum_{i} \frac{1}{\mu_{i}}b_{\mu,i}a_{\mu,i}\delta^{i} +\sum_{i} {g}_\mu(A_\mu,B_\mu)\mu_i \delta_i ),
\end{align}
where the first projection is the same because $\widetilde{\nabla}^{(m)} $ belong to $T\mathcal{P}_+(I)$ whenever it is evaluated in vector fields of $\mathcal{P}_+(I)$, while the second connection must be projected with the extra term,  $\sum_{i} {g}_\mu(A_\mu,B_\mu)\mu_i \delta_i $. A direct consequence of this definition is that

$$A{g}_\mu(B,C)={g}_\mu(\nabla^{(m)}_AB,C )+{g}_\mu(B,\nabla^{(e)}_AC )$$
  for $A=\dfrac{\partial}{\partial \mu_k},B=\dfrac{\partial}{\partial \mu_l},C=\dfrac{\partial}{\partial \mu_s}$ in $\mathcal{P}_+(I)$. This relation allows us to prove the following property:  %( compare with Proposition 1.10.4 in [Ovd geoemtry] ):

\begin{proposition}\label{conexioneslibresdetorsionmye}
    The connections $\nabla^{(m)}$ and $\nabla^{(e)}$ satisfy the condition:
$$A{g}_\mu(B,C)={g}_\mu(\nabla^{(m)}_AB,C )+{g}_\mu(B,\nabla^{(e)}_AC ),$$    for any $A,B,C$ vector field in $\mathcal{P}_+$. In addition, $\nabla^{(m)}$ and $\nabla^{(e)}$ are torsion-free.
\end{proposition}
\begin{proof}
For the first claim, we just note that we have proved for a basis of local vector fields. Now, we will use the fact on $A=\dfrac{\partial}{\partial \mu_k},B=\dfrac{\partial}{\partial \mu_l},C=\dfrac{\partial}{\partial \mu_s}$ and extend (via Leibniz rule for the connections and the vector fields) as $C^\infty(\mathcal{P}_*(I))$-module. For this, consider $b$ a smooth function in $\mathcal{P}_+(I)$, and we will prove the relation for $bA$ as follows:
$$
b(A{g}_\mu(B,C))={g}_\mu(b\nabla^{(m)}_AB,C )+{g}_\mu(B,b\nabla^{(e)}_AC )={g}_\mu(\nabla^{(m)}_{bA}B,C )+{g}_\mu(B,\nabla^{(e)}_{bA}C ); $$
and in a similar way we will proved for $bB$:
\begin{align*}
 A{g}_\mu(bB,C)&=A(b{g}_\mu(B,C))=(Ab){g}_\mu(B,C)+bA{g}_\mu(B,C) \\ 
 &=(Ab){g}_\mu(B,C)+b({g}_\mu(\nabla^{(m)}_AB,C )+{g}_\mu(B,\nabla^{(e)}_AC ))\\
 &={g}_\mu((Ab)B+b\nabla^{(m)}_AB,C )+{g}_\mu(bB,\nabla^{(e)}_AC )\\
 &={g}_\mu(\nabla^{(m)}_AbB,C )+{g}_\mu(bB,\nabla^{(e)}_AC ),
\end{align*}
yielding the desired result. For $bC$, the verification is the same as for $bB$. Finally,  extend the relation linearly and the relation holds.

For the second claim, recall that the torsion of a connection is the tensor determined by $\nabla_{A}B-\nabla_{B}A-[A,B]$, so we must verify that $\nabla_{A}B-\nabla_{B}A=[A,B]$, for $m$-connection and $e$-connection.
    From the defining relation \eqref{eq:m-p+conn}, we just note that 
    %for a connection to be torsion-free it must satisfy  where $[A,B]$  is the Lie bracket. First let's see for the connection $\nabla^{(m)}$ as defined in ~\eqref{mconexion}, apply to vector fields 
    for any vector fields $A=\sum_i a^{i}\frac{\partial}{\partial e^{i}}$ y $B=\sum_j  b^{j}\frac{\partial}{\partial e^{j}}$, we get:
    $$
        \nabla^{(m)}_{A}B-\nabla^{(m)}_{B}A=\sum_{i,j}a^{i}\frac{\partial b_{j}}{\partial e^{i}}-b^{j}\frac{\partial a_{i}}{\partial e^{j}}, %\quad %\mbox{and}\quad         \nabla^{(e)}_{A}B-\nabla^{(e)}_{B}A=\sum_{i,j} a^{i}\frac{\partial b_{j}} {\partial e^{i}}-b^{j}\frac{\partial a_{i}}{\partial e^{j}}.
    $$ 
which is the same as $[A,B]$. Simlarly for the relation in \eqref{eq:e-p+conn}, then we get torsion free in both cases. % as it was showed in (ref lie bracket). %Then, both $\nabla^{(m)}$ as $\nabla^{(e)}$ are torsion-free.
\end{proof}

\begin{rmk}
The previous fact is a well known fact and a key starting point on the theory of information geometry these connections. A proof can be found in \cite[Proposition~1.10.4]{ovidiu}. All the proofs known in the literature use Christhoffel symbols, but our proof avoid this geometrical notion and show that depends only on local direction because the identity $A{g}_\mu(B,C)={g}_\mu(\nabla^{(m)}_AB,C )+{g}_\mu(B,\nabla^{(e)}_AC )$ is tensorial
\end{rmk}

%Therefore, the geometric structure of $(\mathcal{M}_{+}(I),\textbf{g}_{\mu}, \nabla^{(e)}, \nabla^{(m)})$ is dual structure with connections are torsion-free.

%deberia ser un resumen (organizado y completo...inlcuyendo algunas graficas como en la presentacion de la tesis) del capitulo 1 de la tesis

\section{Geometric concepts for information geometry}

Motivated by the geometry on  statistical models, the Fisher metric and dual connections, we will define {\it statistical manifolds} as abstraction of this structures for any Riemannian manifold.  The main goal of this section is to present the difference between statistical model and  statistical manifold. The technique we will use is to endow those manifold with a particular geometric structure and verify that such structures coincide with the statistical model.

%=====include?======\\
%The purpose of this section is to define statistical manifold and its geometric structure in addition to the metric. We will expand our study in terms of geometric information, taking as reference \cite{amariMeth} and \cite{jost}, observing that it is possible to obtain more connections, through the convex combination of the $e$-connection and the $m$-connection. Furthermore, these connections together with the Fisher metric allow us to construct the $3$-tensor of \textbf{Amari-Chentsov}, with which, we give way to the statistical manifold.

%deberia ser un resumen (con pruebas mas cortas...pero enfatizando la novedad de libre de coordenadas) del capitulo 2 de la tesis

\subsection{Statistical models and dual structure}
This section takes as reference \cite[section 3.1]{amariMeth} and \cite[section 4.2]{jost} to study the dual structure and the notion of torsion-free connections for any Riemannian manifold extending the notion of statistical models previously defined.

\begin{definition}\label{def:dual}
    Two affine connections $\nabla^{(1)}$ and $\nabla^{(-1)}$ on a Riemannian manifold $(M,g)$ are called dual connection if for any three vector fields they satisfy \begin{equation}\label{identidaddual}
    Zg(X,Y):=g(\nabla_{Z}^{(1)}X, Y)+g(X,\nabla_{Z}^{(-1)}Y).
\end{equation}
In this case, the triple $(g,\nabla^{(1)},\nabla^{(-1)})$ is called dual structure on $M$.
\end{definition}

%In summary, on statistical models determined by the exponential family and the mixture family; respectively characterized by parallel transport that in turn induces connections ${\nabla}^{(e)}$ and ${\nabla}^{(e)}$. To study how the above are related, it is necessary to analyze the following definition:

%\begin{definition}[\cite{jost} Definition 4.1]  
 %A manifold $M$ together with a Riemannian metric $g$ and two connections $\nabla$ and $\nabla^{*}$ on a differentiable manifold $M$, are dual to each other with respect to g in the sense, that they are called mutually dual if they satisfy the identity:
% A triple $(g, \nabla , \nabla^{*})$  is called a \textbf{dualistic structure on $M$}.
%\end{definition}

%In particular, note that any statistical model $(g,p)$ on a manifold, induce  a dual structure with the connections $\nabla^{(m)}$ and $\nabla^{(e)}$. % are mutually dual with respect to $\textbf{g}_{\mu}$, since it satisfies that:
%\begin{equation}\label{dualidadem}
 %   Z\textbf{g}_{\mu}(A,B)=\textbf{g}_{\mu}(\nabla^{(e)}_{Z}A,B)+\textbf{g}_{\mu}(A,\nabla_{Z}^{(m)}B)
%\end{equation}\\
Recall that the torsion of an affine connection $\nabla$ is $\mathrm{Tor}(X,Y)=\nabla_XY-\nabla_YX-[X,y]$, and is called torsion-free if $\mathrm{Tor}(X,Y)=0$. A Riemannian manifold $(M,g)$ is called {\bf statistical manifold} if it is endowed with a pair of torsion-free dual connections $(\nabla^{(1)},\nabla^{(-1)})$.  

\begin{lemma}\label{unicidadConexion}
Let $(M, g)$  be a Riemannian manifold and $\nabla^{(1)}$ an affine connection, then there exists a unique  dual connection $\nabla^{(1)}$ with respect to $g$.   \end{lemma}

\begin{proof}
The existence of  $\nabla^{(-1)}$ is a direct consequence of the identity  \eqref{identidaddual} defining the dual connection. Now, consider $\Bar{\nabla}$ and $\nabla^{(-1)}$ dual connections to  $\nabla^{(1)}$ with respect to $(M,g)$, that is, for all tangent vectors  $X,Y,Z$ in $M$, we have     \begin{align*}g(X,Y)&=g(\nabla^{(1)}_{Z}X,Y)+g(X,\Bar{\nabla}_{Z}Y)\\  Zg(X,Y)&=g(\nabla^{(1)}_{Z}X,Y)+g(X,\nabla^{(-1)}_{Z}Y);
    \end{align*}
in particular we get  $0=g(X,(\Bar{\nabla}-\nabla^{(-1)})_{Z}Y)$, which finally yields $\Bar{\nabla}=\nabla^{(-1)}$. 
    
\end{proof}

We will present a summary of well known facts of geometry of dual connections, but we state them in a general version and give coordinate-free proofs that are not available in the literature.
\begin{theorem}\label{estructuradualconexiAlpha}
Let $(M,g)$ be a Riemannian manifold, and  $(\nabla^{(1)},\nabla^{(-1)})$ any two connections in $M$.
\begin{enumerate}
\item $(\nabla^{(1)},\nabla^{(-1)})$ is dual structure if and only if $(\alpha\nabla^{(1)}+\beta\nabla^{(-1)},\beta\nabla^{(1)}+\alpha\nabla^{(-1)})$ is dual structure for any combination such that $\alpha+\beta=1$.
    \item $(\nabla^{(1)},\nabla^{(-1)})$ are torsion-free if and only if $(\alpha\nabla^{(1)}+\beta\nabla^{(-1)},\beta\nabla^{(1)}+\alpha\nabla^{(-1)})$ are torsion-free for any combination such that $\alpha+\beta=1$,    
    \item If $(\nabla^{(1)},\nabla^{(-1)})$ are dual and torsion free, then $2\nabla^{(0)}=\nabla^{(1)}+\nabla^{(-1)}$,
    \item whenever $2\nabla^{(0)}=\nabla^{(1)}+\nabla^{(-1)}$, then $\nabla^{(1)}$ is torsion free if and only if $\nabla^{(-1)}$ is also torsion free,
    
\end{enumerate}
    
\end{theorem}

\begin{proof}
Before we give the proofs of each item, we must remark that any $\mathbb{R}$-linear combination of affine connection is again affine connection, just because all the defining condition (as in Remark~?) are preserved by linearity.\\
{\it Proof of (1):}First assume that $(\nabla^{(1)},\nabla^{(-1)})$ are dual structure, then by linearity of the metric we get
\begin{align*}
g&=((\alpha\nabla^{(1)}+\beta\nabla^{(-1)})_XY,Z)+g(Y,(\beta\nabla^{(1)}+\alpha\nabla^{(-1)})_XZ)\\
&=\alpha(g(\nabla^{(1)}_XY,Z)+g(Y,\nabla^{(-1)}_XZ))+\beta(g(\nabla^{(-1)}_XY,Z)+g(Y,\nabla^{(1)}_XZ))=Xg(Y,Z).
\end{align*}

For the other direction, just note that 
\begin{align*}
\nabla^{(1)}&=\tilde{\alpha}(\alpha\nabla^{(1)}+\beta\nabla^{(-1)})+\tilde{\beta}(\beta\nabla^{(1)}+\alpha\nabla^{(-1)})\\
\nabla^{(-1)}&=\tilde{\beta}(\alpha\nabla^{(1)}+\beta\nabla^{(-1)})+\tilde{\alpha}(\beta\nabla^{(1)}+\alpha\nabla^{(-1)})
\end{align*}
with $\tilde{\alpha}=\dfrac{\alpha}{\alpha-\beta}$ and $\tilde{\beta}=\dfrac{-\beta}{\alpha-\beta}$. Note that we also have that $\tilde{\alpha}+\tilde{\beta}=1$, and the results follows from the previous claim.\\

{\it Proof of (2):} This verification follows the same argument as previous item.\\
%Note that conditions in equation \eqref{def:conn} are preserved by any linear combination, then $\alpha \nabla^{(1)}+ \beta \nabla^{(-1)}$ for any scalar $\alpha,\beta$ is again a connection. Now, let see the torsion of a this combination:
%\begin{align*}
%    (\alpha \nabla^{(1)}+ \beta \nabla^{(-1)})_{X}Y-(\alpha \nabla^{(1)}+ \beta \nabla^{(-1)})_{Y}X&= \alpha(\nabla^{(1)}_{X}Y)+\beta(\nabla^{(-1)}_{X}Y)-\alpha(\nabla^{(1)}_{Y}X)-\beta(\nabla^{(-1)}_{Y}X)\\
%    &=\alpha(\nabla^{(1)}_{X}Y-\nabla^{(1)}_{Y}X)+\beta(\nabla^{(-1)}_{X}Y-\nabla^{(-1)}_{Y}X)\\
%    &=\alpha([X,Y])+\beta([X,Y])=(\alpha+\beta)[X,Y]
%\end{align*}
%as this is a convex combination $\alpha+\beta=1$ then it is torsion-free.
        
{\it Proof of (3):}Just note that for $\alpha=\dfrac{1}{2}=\beta,$ we get $(\alpha\nabla^{(1)}+\beta\nabla^{(-1)},\beta\nabla^{(1)}+\alpha\nabla^{(-1)})$ is dual structure and torsion-free (previous items). Indeed, in this case we have 
$$\alpha\nabla^{(1)}+\beta\nabla^{(-1)}=\beta\nabla^{(1)}+\alpha\nabla^{(-1)},$$
which means that it is {\it self-dual}, or in other words it is a metric connection. By fundamental theorem in Riemannian geometry, we get that $\alpha\nabla^{(1)}+\beta\nabla^{(-1)}$ is the Levi-Civita connection, and the claim holds.\\

{\it Proof of (4):} An easy verification from the identity $2\nabla^{(0)}=\nabla^{(1)}+\nabla^{(-1)}$, leads us to note that
$$\nabla^{(1)}_XY-\nabla^{(1)}_YX=2\nabla^{(0)}_XY-\nabla^{(-1)}_XY-(2\nabla^{(0)}_YX-\nabla^{(-1)}_XY)=2[X,Y]-(\nabla^{(-1)}_XY-\nabla^{(-1)}_YX),$$
which yields that $\nabla^{(1)}_XY-\nabla^{(1)}_YX-[X,Y]=[X,Y]-(\nabla^{(-1)}_XY-\nabla^{(-1)}_YX)$, and the claim holds.
\end{proof}

In what follows, we work only with this type of combinations, in particular we will consider the family of  \textbf{$\alpha$-connection}, with  $\alpha\in [-1,1]$, as:
%Veamos que es posible deducir infinitas conexiones, mediante la combinaci\'{o}n convexa de ellas, para nuestro caso con $\nabla^{(1)}$ y $\nabla^{(-1)}$. D

    \begin{equation}\label{eq:a-conn}
       ( \nabla^{(\alpha)}=\frac{1+\alpha}{2}\nabla^{(1)}+\frac{1-\alpha}{2}\nabla^{(-1)}, \nabla^{(-\alpha)}=\frac{1-\alpha}{2}\nabla^{(1)}+\frac{1+\alpha}{2}\nabla^{(-1)})
    \end{equation}

or in its equivalent version:
%\begin{align*}
 %     {\nabla}^{(\alpha)}=& \frac{1}{2}{\nabla}^{(-1)}-\frac{\alpha}{2}{\nabla}^{(-1)}+\frac{1}{2}{\nabla}^{(1)}+\frac{\alpha}{2}{\nabla}^{(1)}=  \frac{1}{2}{\nabla}^{(-1)}+\frac{1}{2}{\nabla}^{(-1)}-\frac{1}{2}{\nabla}^{(-1)}-\frac{\alpha}{2}{\nabla}^{(-1)}+\frac{1}{2}{\nabla}^{(1)}+\frac{\alpha}{2}{\nabla}^{(1)}\\    = & {\nabla}^{(-1)}+\frac{1}{2}({\nabla}^{(1)}-{\nabla}^{(-1)})+\frac{\alpha}{2}({\nabla}^{(1)}-{\nabla}^{(-1)}) =  {\nabla}^{(-1)}+(\frac{1}{2}+\frac{\alpha}{2})({\nabla}^{(1)}-{\nabla}^{(-1)})
%\end{align*}
%As\'{i}, se reescribe como:
\begin{align}\label{alphaconexion}
    {\nabla}^{(\alpha)}= {\nabla}^{(-1)}+(\frac{1+\alpha}{2})({\nabla}^{(1)}-{\nabla}^{(-1)}),\quad {\nabla}^{(-\alpha)}= {\nabla}^{(1)}-(\frac{1+\alpha}{2})({\nabla}^{(1)}-{\nabla}^{(-1)}).
\end{align}
As a direct consequence of the previous theorem, we get:
\begin{corollary}\label{cor4.5}
 If  $(g,\nabla^{(1)}, \nabla^{(-1)})$ is a torsion-free dual structure in $M$, then  $(g,\nabla^{(\alpha)}, \nabla^{(-\alpha)})$ is also torsion-free dual structure, for any $-1\leq \alpha \leq 1$. and we get that, 
 \begin{equation}\label{Levi-Civita-alphaconexi}
    \nabla^{(0)}=\frac{1}{2}(\nabla^{(\alpha)}+\nabla^{(-\alpha)}).
\end{equation}

\end{corollary}
%Tambi\'{e}n tenemos $\nabla^{(\alpha)}$ como combinaci\'{o}n convexa de $\nabla^{(-1)}$ y $\nabla^{(1)}$  entonces $\nabla^{(\alpha)}$ y $\nabla^{(-\alpha)}$ son libres de torsi\'{o}n y 

\begin{example}
Using the notation:
\begin{align*}
    \nabla^{(1)}:=\cc^{(e)} \qquad  \nabla^{(-1)}:=\cc^{(m)}
\end{align*}
we recall that the statistical model $(\mathcal{P}_{+}(I),{g}, \nabla^{(-1)}, \nabla^{(1)})$ is also a statistical manifold. Additionally, previous definition yields that $\alpha$-connections  are torsion-free and dual structure. Furthermore, using ~\eqref{econexion} and ~\eqref{mconexion},  applied to $B$ in the direction of $A$ (in $\mathcal{P}_{+}(I)$) gives local representation as:
\begin{align}\label{conexalpha}
    {\nabla}^{(\alpha)}_{A}B |_{\mu}=\Big(\mu, \sum_{i}\Big(\frac{\partial b_{i}}{\partial a_{\mu}}(\mu)-\frac{1+\alpha}{2} {g}_{\mu}(A,B)\Big)\delta^{i}\Big),\quad {\nabla}^{(-\alpha)}_{A}B |_{\mu}=\Big(\mu, \sum_{i}\Big(\frac{\partial b_{i}}{\partial a_{\mu}}(\mu)-\frac{1-\alpha}{2} {g}_{\mu}(A,B)\Big)\delta^{i}\Big).
\end{align}

\end{example}

As final remark in this section is that the usual prove of the previous results lie on Christoffel symbols for the two dual connection on statistical manifolds, however we give a coordinate free proof of these facts and even in the more general setting of statistical structure \cite[cf. Section 4.2]{jost}.

\subsection{Tensor Amari-Chentsov}\label{estructuraestadistica}
We now want to measure the difference between two dual structures $(\nabla^{(1)},\nabla^{(-1)})$, in explicit way we want to compute $\mathcal{T}=\nabla^{(-1)}-\nabla^{(1)}:\mathfrak{X}^2(M)\to \mathfrak{X}(M)$. We can study this difference by using  the metric tensor $g$, i.e we define the following tensor:
 or $T:\mathfrak{X}^3(M)\to C^\infty(M)$:
    \begin{equation}\label{ecu-Tensor}
      T(X,Y,Z)={g}(\nabla_{X}^{(-1)}Y-\nabla_{X}^{(1)}Y, Z)={g}(\mathcal{T}(X,Y),Z)
       \end{equation}

Note the special case of $\nabla^{(0)}$ of Levi-Civita connection, this tensor vanishes identically. In the Literature, $T$ (simialry $\mathcal{T}$) is known as {\bf Amari-Chentsov} tensor.
\begin{example}
    
Observe, ${\nabla}_{A}^{(-1)}B|_{\mu}=\frac{\partial b}{\partial a_\mu}(\mu)$ and ${\nabla}_{A}^{(1)}B|_{\mu}=\frac{\partial b}{\partial a_{\mu}}(\mu)-{g}_{\mu}(A,B)$, the difference between them on   $A \sim (\mu, a_{\mu}), B \sim (\mu, b_{\mu})\in T_\mu \mathcal{M}_+(I)$ is:
\begin{equation}\label{diferencia-conexiones}
     \mathcal{T}(A,B):={\nabla}_{A}^{(-1)}B|_{\mu}-{\nabla}_{A}^{(1)}B|_{\mu}={g}_{\mu}(A,B)=\sum_{i\in I}\frac{1}{\mu_{i}}a_{i}b_{i} 
\end{equation}
Using the Fisher metric with respect to other tangent vector  $C=(\mu,c_\mu)$ on $\mu\in \mathcal{M}_+(I)$, the relation \eqref{ecu-Tensor} yields \cite[section 2.5.1]{jost}:
\begin{align}\label{tensorAmariChentsovCoor}
    \mathbf{T}_{\mu}(A_{\mu}, B_{\mu}, C_{\mu})= \sum_{i\in I} \mu_{i} \frac{a_{\mu,i}}{\mu_{i}}\frac{b_{\mu,i}}{\mu_{i}}\frac{c_{\mu,i}}{\mu_{i}}.
\end{align}
In this way the Amari-Chentsov in $\mathcal{M}_+(I)$ is: 
\begin{align*}
    T_{\xi} &=E_{{p}}\Big[\frac{\partial}{\partial \xi_{i}}\log p \frac{\partial}{\partial \xi_{j}}\log p \frac{\partial}{\partial \xi_{k}}\log p\Big]\\
    %&=\int_{I}\frac{\partial}{\partial \xi^{i}}\log p(x; \xi) \frac{\partial}{\partial \xi^{j}}\log p(x; \xi)\frac{\partial}{\partial \xi^{k}}\log p(x; \xi) p(x; \xi)dx
\end{align*}
%Sobre $I$.
\end{example}

\begin{proposition}\label{dual-estructura}\cite[Theorem 4.1]{jost} The Amari-Chentsov tensor $T$ from a torsion-free dual structure  $(g, \nabla^{(1)}, \nabla^{(-1)})$ is a symmetric 3-tensor.
 \end{proposition}
We remark that all proof of this claim lies on the Christoffel symbols, but here we will do it globally. 
 \begin{proof}
First, by using the dual structure, we get:
\begin{align*}
    T(X,Y,Z) &=g(\nabla^{(-1)}_{X}Y-\nabla^{(1)}_{X}Y,Z)-g(\nabla^{(1)}_{X}Y,Z)=Xg(Y,Z)-g(Y,\nabla^{(1)}_{X}Z)-(Xg(Y,Z)-g(Y,\nabla^{(-1)}_{X}Z))\\
    &=g(Y,\nabla^{(-1)}_{X}Z)-g(Y,\nabla^{(1)}_{X}Z)=T(X,Z,Y).
\end{align*}
Also note that the torsion-free condition implies that:
$$
    T(X,Y,Z)=g(\nabla^{(-1)}_{X}Y,Z)=g(([X,Y]+\nabla^{(-1)}_{Y}X)-([X,Y]+\nabla^{(1)}_{Y}X),Z)=T(Y,X,Z).
$$
These two relations show that $T$ is symmetric on the three components. It remains to verify that it is tensorial, but this follows directly from the fact that:
$$T(X,Y,fZ)=g(\nabla^{(-1)}_{X}Y-\nabla^{(1)}_{X}Y,fZ)=fg(\nabla^{(-1)}_{X}Y-\nabla^{(1)}_{X}Y,Z)=fT(X,Y,Z).$$ From this identity, we conclude: 
$$fT(X,Y,Z)=fT(X,Z,Y)=T(X,Z,fY)=T(X,fY,Z)$$
and similar for $fT(X,Y,Z)=T(fX,Z,Y)$.
 \end{proof}

\begin{definition}\label{defVariedaEstadis}\cite [Definition 4.2]{jost}
    A \textbf{statistical structure} on a manifold $M$ consists of a metric  $g$ and a \textit{3-tensor} $T$  that is symmetric in all arguments. %A \textbf{statistical manifold} is a manifold $M$ equipped with a statistical structure $(M, g, T)$.
\end{definition}

Indeed, both structures are the same, as it is showed in the following results. We give complete and explicit construction to show the dependence on the Levi-Civita connection.

\begin{proposition}\label{corola-gytdefEstruDual} \cite[Theorem 4.2]{jost}  
   Each statistical structure $(M,g,T)$ induces a statistical manifold   $(M,g,\nabla^{(1)}, \nabla^{(-1)})$, i.e dual and torsion-free. % Una m\'{e}trica $g$ y un 3-tensor sim\'{e}trico $T$, tiene asociado una estructura dual $({\bf g},\nabla,\nabla^*)$ con conexiones libres de torsi\'{o}n.
\end{proposition}

\begin{proof} The idea is to use an auxiliar connection to construct two connections, and conditions from the auxiliar translate to the new ones. Let us denote by $\Bar{\nabla}$ an auxiliary connection and define two new connections:
     \begin{align}\label{defconexTensor}
          \nabla_{Z}X = \Bar{\nabla}_{Z}X-\frac{1}{2}\mathcal{T}(Z,X),\quad \nabla^*_{Z}X = \Bar{\nabla}_{Z}X+\frac{1}{2}\mathcal{T}(Z,X).
     \end{align}
It is a straightforward computation to verify that are $\mathbb{R}$-linear. So it remains to study the behaviour with respect to product with $f\in C^\infty(M)$. We give the proof for $\nabla$ and for $\nabla^*$ the computation is the same.
\begin{align*}
        \nabla_{Z}(fX) &=\Bar{\nabla}_{Z}(fX)-\frac{1}{2}\mathcal{T}(Z,fX)=
        f(\Bar{\nabla}_{Z}X-\frac{1}{2}\mathcal{T}(Z,X))+Z(f)X=f\nabla_ZX+Z(f)X \\
        \nabla_{fZ}(X)&=\Bar{\nabla}_{fZ}(X)-\frac{1}{2}\mathcal{T}(fZ,X)=f(\Bar{\nabla}_{Z}X-\frac{1}{2}\mathcal{T}(Z,X))=f\nabla_ZX .
    \end{align*}
Then, $\nabla$ and  $\nabla^*$ are affine connection. 
Now, we study the torsion tensor of $\nabla$:
\begin{align*}
     \nabla_{Z}X-\nabla_{X}Z-[Z,X] &=\Bar{\nabla}_{Z}X-\frac{1}{2}\mathcal{T}(Z,X)-(\Bar{\nabla}_{X}Z-\frac{1}{2}\mathcal{T}(X,Z))-[Z,X]\\
      &=\Bar{\nabla}_{Z}X+\Bar{\nabla}_{X}Z-\frac{1}{2}(\mathcal{T}(X,Z)-\mathcal{T}(Z,X))-[Z,X]\\
      &=\Bar{\nabla}_{Z}X-\Bar{\nabla}_{X}Z-[Z,X].
\end{align*}
And a similar computation for the torsion of the tensor $\nabla^*$ gives us $\nabla^*_{Z}X-\nabla^*_{X}Z-[Z,X] =\Bar{\nabla}_{Z}X-\Bar{\nabla}_{X}Z-[Z,X]$. Thus, if we impose that $\Bar{\nabla}$ is torsion-free, we also have same property for both $\nabla$ and $\nabla^*$.

The duality condition follows a similar argument, just by noticing that.
\begin{align*}
    g(\nabla_{Z}X,Y)+g(X,\nabla^{*}_{Z}Y) &=[g(\bar{\nabla}_{Z}X,Y)-\frac{1}{2}T(Z,X,Y)]+[g(X,\bar{\nabla}_{Z}Y)+\frac{1}{2}T(Z,Y,X)]\\ &=g(\bar{\nabla}_{Z}X,Y)+g(X,\bar{\nabla}_{Z}Y)-\frac{1}{2}T(Z,X,Y)+\frac{1}{2}T(Z,Y,X)\\    &=g(\bar{\nabla}_{Z}X,Y)+g(X,\bar{\nabla}_{Z}Y).
\end{align*}
Therefore, if it is also assumed that $\Bar{\nabla}$ is a metric connection, we have:
\begin{align*}
    g(\nabla_{Z}X,Y)+g(X,\nabla^{*}_{Z}Y) &=Zg(X,Y).
\end{align*}
From previous discussion, if we choose $\Bar{\nabla}$ as the Levi-Civita connection for $(M,g)$, then the symmetric tensor $\mathcal{T}$ yields (from \eqref{defconexTensor}) $(M,g, \nabla^{(1)}=\nabla, \nabla^{(-1)}=\nabla^{*})$ a statisical manifold.
\end{proof}

Furthermore, these two procedures are inverse of each other, since given the dual structure $(\nabla,\nabla^*)$ regarding the metric $g$ we have $\mathcal{T}=\nabla^{*}-\nabla$, the statistical structure ($g, \mathcal{T}$) and this produces the dual structure, as follows:
\begin{align*}
    \nabla'=\nabla^{(0)}-\frac{1}{2}\mathcal{T}=\frac{1}{2}\nabla+\frac{1}{2}\nabla^{*}-\frac{1}{2}(\nabla^{*}-\nabla)=\nabla
\end{align*}
\begin{align*}
    \nabla'^{*}=\nabla^{(0)}+\frac{1}{2}\mathcal{T}=\frac{1}{2}\nabla+\frac{1}{2}\nabla^{*}+\frac{1}{2}(\nabla^{*}-\nabla)=\nabla^{*}
\end{align*}
On the other hand: $(g,\mathcal{T})$ produces $(\nabla, \nabla^{*})=(\nabla^{(0)}-\frac{1}{2}\mathcal{T}, \nabla^{(0)}+\frac{1}{2}\mathcal{T})$ in turn gives $$\mathcal{T}'=(\nabla^{(0)}+\frac{1}{2}\mathcal{T})-(\nabla^{(0)}-\frac{1}{2}\mathcal{T})=\mathcal{T}$$

In conclusion,

\begin{theorem}\label{thm:3ten=dualst}
For any Riemannian metric $(M,g)$, there is a one-to-one relation between torsion-free dual structures (statistical manifold) and statistical structure.
\end{theorem}

\subsection{Isostatistical immersion}\label{sec:isoest}

In the previous section, we conclude that there is a one-to-one relation between statistical structures and statistical manifolds, but both of them are motivated by the statistical model $(\mathcal{P}_+(I),g)$. A natural question arises when we want to see any statistical structure as a statistical model. 
%This question is that if each statistical manifold is also statistical model. 
For this question, we should consider the following notion:

\begin{definition} \cite[Definition 4.9]{jost}  
    Let $h$ be a smooth application of a statistical manifold $(M_{1}, g_{1},T_{1})$ to a statistical manifold $(M_{2}, g_{2},T_{2})$. The map $h$ will be called \textbf{isostatistical immersion} if it is an immersion of $M_{1}$ into $M_{2}$ such that $g_{1}=h^{*}(g_{2})$ and $T_{1}=h^{*}(T_{2})$.
    \end{definition}

A first result related to this notion is an original extension of the statement in \cite[Lemma~4.6]{jost} for Fisher metric. It is important to highlight that \cite[Lemma~4.6]{jost} just prove item a., and in this note we give a proof for the rest of the claims:
%{\color{red}re-organizar enunciado y prueba}

\begin{proposition}\label{prop:nueva1}
    Let $\Omega$ be a measurable space and  $p_{i}:\Omega \times M_{i}\longrightarrow [0,1]$ measurable functions with $\int_{\Omega \times M_i}p_i(x;\xi_i) dx=1$ (for $i=1,2$), such that $h:(M_{1}, g_{1},T_{1})\longrightarrow (M_{2}, g_{2},T_{2})$ is an isostatistical immersion between this two statistical structures, then we have:
\begin{itemize}
    \item[a.] If $g_1$ and $g_2$ are Fisher metrics, then $h^{*}p_{2}(x;\xi_{2})=p_{1}(x;\xi_{1})$.
    \item[b.] If $h^{*}p_{2}(x;\xi_{2})=p_{1}(x;\xi_{1})$, and $g_2$ is Fisher metric, then $g_{1}$ It is also a Fisher metric.
    \item[c.] If $h^{*}p_{2}(x;\xi_{2})=p_{1}(x;\xi_{1})$, and  $T_{2}$ is Amari–Chentsov tensor, then $T_{1}$ is Amari–Chentsov tensor.
\end{itemize}

\end{proposition}

\begin{proof}

As we comment before the statement of the proposition, we just prove the second and third claim:
%    Given $h$ is an isostatistical immersion, we have:\begin{align*}
%    g_{1}(V_{1}, V_{2}) = &  \int_{\Omega\times M_{1}}\frac{\partial}{\partial h_{*}(V_{1})}\log h^{*} p_{2}(x;h(\xi_{1}))\frac{\partial}{\partial h_{*}(V_{2})}\log h^{*} p_{2}(x;h(\xi_{1}))h^{*}p_{2}(x;h(\xi_{1}))dx \\
%    =& \int_{\Omega \times M_1}\frac{\partial}{\partial V_{1}} \log p_{1}(x;\xi_{1}) \frac{\partial}{\partial V_{2}}\log p_{1}(x;\xi_{1}) p_{1}(x; \xi_{1})dx,
%    \end{align*}
%    then, if $h$ is local diffeomorphism, we have 
%$h^{*}p_{2}(x;\xi_{2})=p_{1}(x;\xi_{1})$.
\begin{itemize}
    \item[b.] Given $h^*g_2=g_1$ and Fisher metric $g_{2}$, we have,
$$g_{1}(V_{1}, V_{2})=E_{h^{*}p_2}\Big[\frac{\partial}{\partial V_{1}}\log h^{*}p_{2}(x; \xi_{2})\frac{\partial}{\partial V_{2}}\log h^{*}p_{2}(x;\xi_2)\Big]=E_{p_1}\Big[\frac{\partial}{\partial V_{1}}\log p_{1}(x; \xi_{1})\frac{\partial}{\partial V_{2}}\log p_{1}(x;\xi_1)\Big],$$
then $g_{1}$ is Fisher metric.

\item[c.] Given $h^*T_2=T_1$ and $T_1$ Amari–Chentsov tensor, we have
\begin{align*}
     T_{1}(V_1, V_2,V_3)  &=E_{h^{*}p_{2}}\Big[\frac{\partial}{\partial V_{1}}\log h^{*}p_{2}(x;\xi_{2}) \frac{\partial}{\partial V_{2} }\log h^{*}p_{2}(x;\xi_{2}) \frac{\partial}{\partial V_3 }\log h^{*}p_{2}(x;\xi_{2})\Big]\\ 
     &=E_{p_{1}}\Big[\frac{\partial}{\partial V_{1}}\log p_{1}(x;\xi_{1}) \frac{\partial}{\partial V_{2} }\log p_{1}(x;\xi_{1}) \frac{\partial}{\partial V_3 }\log p_{1}(x;\xi_{1})\Big],
\end{align*}

then $T_{1}$ is Amari–Chentsov tensor.
\end{itemize}
\end{proof}
Now, we present (without proof) the representation theorem for statistical manifolds that gives a positive answer to the initial question of this section:

\begin{theorem}(Existence of isostatistical immersion \cite[Theorem 4.10]{jost})
Any smooth compact statistical manifold $(M,g,T)$ admits an isostatistical immersion $h$ into the statistical manifold $(\mathcal{P}_{+}(I), \mathbf{g},\mathbf{T})$, for some finite set $I$. %Hence any statistical manifold can be represented by the Fisher metric and the Amari-Chentsov tensor, but in another space, {\color{red}as is a statistical model.}
\end{theorem}

Consequently, although the metric $g$ in $M$ is not a Fisher metric, it is isometric to $g_h$ in $h(M)$ that is isostatistical to $\mathbf{g}$ in $\mathcal{P}_{+}(I)$; however we do not have a {\it fauthfully } representation because there are geoeomtrical structure that are not preserved by this representation as we will comment at the end of the following section.

\begin{rmk}\cite[Theorem 4.10]{jost}
   Any non-compact statistical manifold $(M,g,T)$ admits an immersion $h$ into the space $\mathcal{P}_{+}(\mathbb{N})$ of all positive probability measures on the set $\mathbb{N}$ of all natural numbers such that $g$ is isometric to the induced Fisher metric on $h(M)$ as finite dimensional submanifold in $\mathcal{P}_{+}(\mathbb{N})$.
\end{rmk}

\subsection{Canonical divergence}\label{sec:can.div}
A divergence function is a non negative function $D:M\times M\to \R$ so that $D|_\Delta=0$\footnote{By $\Delta$ we mean the diagonal manifold inside $M\times M$} and for any two vector fields $X,Y$ in $M$ we get $X^1Y^2D|_\Delta >0$ where the superscript $X^1$ and $Y^2$ represent the lifting to $M\times \{q\}$ and $\{p\}\times M$ respectively.  We will say that a divergence function {\it generates} the structure $(g,\cc^{(1)},\cc^{(-1)},T)$ if $g=g^{(D)}$, $(\cc^{(1)}=\cc^{(D)},\cc^{(-1)}=\cc{^{(D^*)}})$ and $T=T^{(D)}$ for the following relations:

\begin{align}    \label{metricdiver}
    g^{(D)}(V,W)|_{p}&:=-D(V\| W)(p),\\ \label{cc-diver}%$=-V_{p}W_{q}D(p\|q)|_{p=q}$
    g^{(D)}(\nabla_{V}^{(D)}W,Z)&:=-D(VW\|Z), \\ \label{dc-diver}
    g^{(D)}(\nabla_{V}^{(D^{*})}W,Z)&:=-D^{*}(VW\|Z), \mbox{\ with\ } D^*(p,q)=D(q,p),\\ \label{T-diver}
    T^{D}(X,Y,Z)&:=-D(XY\|Z)+D(Z\|XY).\\ \notag
\end{align}

where we are using the usual notation:

\begin{equation*}
    D(V_{1} \cdots V_{n} \| W_{1} \cdots W_{m})(p) := (V_{1})^{1}  \cdots (V_{n})^{1}(W_{1})^{2}\cdots (W_{m})^{2}D|_{\Delta}.
\end{equation*}
%where $V_{j}$ are seen as the lifting to $M \times \{q\}$ and $W_{j}$ are the liftings to $\{p\} \times M$, that is, we are deriving in different components.

The problem of the existence of divergence function that generates a statistical manifold (and also the statistical structure) is already solved by \cite[Amari et al]{amariMeth}., indeed there exist many of such functions. For more details and references we refer to \cite[section 4.4]{jost}.

As a consequence of the isostatistical immersion theorem, we have the following corollary:
\begin{corollary}\label{Corolario-4.5}
\cite[Corollary 4.5]{jost}
 For any statistical manifold $(M, g,T)$ we can find a divergence $D$ of $M$ which defines $g$ and $T$ by the formulas  ~\eqref{metricdiver}-\eqref{T-diver}.

\end{corollary}

%{\color{red}no se si esto vaya...revisa si vale al pena incluir KL-divergence y otras...lo importante es incluir divergencia canonica}
%We refer the proof to \cite[section 4.5.1]{jost} emphasizing the construction of a specific divergence called Kullback-Leibler, which is found in reference \cite[definition 2.8]{jost}. Its importance is in the use of convex functions determined by the exponential and mixture families, inducing KL-divergence, which when partially derived in turn, induces the dual structure.\\ 

A related question is if there exists a natural choice among all of these divergence functions. The answer comes as {\bf canonical divergence}, that, in addition to the definition of divergence must satisfies: %{\color{red} requirements in section 4.4.2 saying that hold for the mixed and exp coordinates}
\begin{enumerate}
\item $D$ generates the dualistic structure $(g,\nabla^{(1)},\nabla^{(-1)})$,
\item $D$ is one half of the squared Riemannian distance, i.e. $2D(p,q) =d(p,q)^2$, when the statistical manifold is self-dual, namely when $\nabla^{(1)},\nabla^{(-1)}$ are equal and coincide with the Levi–Civita connection,
\item $D$ is the canonical divergence, when $(M,g,\nabla^{(1)},\nabla^{(-1)} )$ is dually flat; this is \textit{Bregman divergence.}\footnote{\textbf{Bergman divergence}\cite[section 4.6]{AmariCichocki} is a more general class of divergences, which includes KL divergence as a special case. It is based on a convex function and can be applied to broader vector spaces beyond probability distributions.}
\end{enumerate}
It was computed that canonical divergence is induced by the geodesics in the following way: On a manifold $M$ which has the associated connection  $\nabla$ concerning the metric $g$, given a pair of (closed enough) points $q,p\in M$, there exists a unique curve $\nabla-$geodesic.
$$\gamma_{q,p}:[0,1]\longrightarrow M,$$ 
which satisfies $\gamma_{q,p}(0)=p$ y $\gamma_{q,p}(1)=q$.
\begin{rmk}
This is equivalent to saying that for each pair of points $q$ and $p$ there exists a unique vector $X(q,p)\in T_{q}M$ that satisfies $\exp_{q}(X(q,p))=p$ where $\exp$ denotes the exponential map associated with $\nabla$.
%~\eqref{mapeoexpo}. %Tomando una parte de la geod\'{e}sica en el intervalo $[t,1]$, se tiene el vector $X(\gamma_{q,p}(t),p)=(1-t)\dot{\gamma}_{q,p}(t)$, usando la geod\'{e}sica inversa $\gamma_{p,q}(t)=\gamma_{q,p}(1-t)$, con estos preliminares se define la divergencia can\'{o}nica como sigue:\\
\end{rmk}

\begin{theorem}\cite[section 4.4.2]{jost}
    Giving an affine connection $\cc$ and a metric $g$ on $M$, it is possible to define the  \textbf{canonical divergence} $D:M\times M\to \mathbb{R}$ associated to  $(g,\cc)$ as:
    \begin{equation}\label{divCanonica}
      D(p\|q)=\int_{0}^{1} t\| \dot{\gamma}_{p,q}(t)\|^{2}dt,  
    \end{equation}
  %      D(p\|q) &:==\int_{0}^{1} \langle X(\gamma_{q,p}(t),p), \dot{\gamma}_{q,p}(t) \rangle dt=\int_{0}^{1} \langle (1-t)\dot{\gamma}_{q,p}(t) , \dot{\gamma}_{q,p}(t) \rangle dt \nonumber \\
   %     D(p\|q) &=\int_{0}^{1}(1-t)\langle \dot{\gamma}_{q,p}(t), \dot{\gamma}_{q,p}(t) \rangle dt=\int_{0}^{1} (1-t)\| \dot{\gamma}_{q,p}(t)\|^{2}dt= \int_{0}^{1} t\| \dot{\gamma}_{q,p}(1-t)\|^{2}dt  \no D^{\nabla}(p \| q)number \\
        %\label{divCanonica}
     
     for $\gamma$ the geodesic with initial point $p$ and ends in $q$. In similar way, it also defines the dual canonical divergence as $D^*(p\|q):=D(q\|p).$
\end{theorem}

%These two functions work as potential functions for the metric and for the connections, indeed it was proved (cf ??) that both divergence defines all the statistical structure on $\mathcal{M}_+(I)$. 
As a direct consequence, for the statistical model, we get:

\begin{corollary}\cite[Theorem 4.8]{jost}
The statistical model $(\mathcal{P}_{+}(I),g,\cc^{(e)},\cc^{(m)},T)$ coincides with 

$(\mathcal{P}_{+}(I),g^{(D)},
\cc^{(D)}, {\dc}^{(D^{*})}, T^{(D)})$.

\end{corollary}

\subsubsection*{The KL-divergence}

A particular case of canonical divergence is the KL-divergence  in $\mathcal{P}_+(I)$ which are defined by
\begin{align}\label{ec-KL}
    D_{KL}(\mu \| \nu)&=\sum_{i\in I}\mu_{i}\log \frac{\mu_i}{\nu_i}.
\end{align}

The verification of this claim can be seen in \cite[Section 4.4.3]{jost}. The idea behind the proof is that for a torsion-free dual structures $(\nabla^{(1)} ,\nabla^{(-1)})$, the formulae in \eqref{divCanonica} can be expressed as: 
$$D(p\|q)=\Psi(p)+\varphi(q)-\vartheta^i(p)\eta_i(q),$$
for suitable function $\Psi,\varphi,\vartheta^i$ and $\eta_i$. The previous description of $D$ coincides with equation \eqref{ec-KL} defining the  KL-divergence. 

The key property at this point is that KL-divergence is related with most of the geometric notions associated to information theory  
%page 57
\cite[section 3.2]{amariMeth}. Indeed, we can relate the Shannon entropy (global information) and the Fisher metric (local information) via the KL-divergence. This claim follows classical fact that KL-divergence recovers the Fisher metric (as divergence in $\mathcal{P}_+$) and 
$$D_{KL}(p\|q)=h(p,q)-h(p)$$
 where $h(p,q)$ is the {\it differential cross-entropy} of $p$ and $q$,  and $h(P)$ is the {\it marginal differential entropy} of $p$. See proof in
%page 206
\cite[section 4.3]{jost}.

\section{The warped product for statistical structure}\label{sec:warped}	

%secciones 3.2 y 3.3 de la tesis

The goal of this section is to extend the statistical geometric structure to the warped product of two statistical manifolds.  In particular, we will study the main geometrical notion presented before, dual structure, symmetric 3-tensor, and divergence functions, under this metric product. As we mention before, the advantage of warped product over a usual cartesian product is that we can modify distances and volumes in one direction of the product, but preserve the angles. Partial results on these goals can be found in \cite{leonard}. The main result in this section is the expression for the lifting of the Amari-Chentsov tensor and the divergence under the warped product.

\begin{rmk}
Indeed, there is the option to modify the length on each direction of the Cartesian product, this is developed by a {\bf double warped product}, but as this is an iteration of the one warped product, the same results presented here are also valid for the double warped product. The formula given for double warped product are similar but tedious and for our purpose there is no additional information.
\end{rmk}

Here we will use same construction in Proposition \ref{prop:lift-conn} for the connections in dual structures on both $B$ and $F$, because this procedure works for any connection. For this reason, we obtain a pair of connections, namely $^f\!\cc^{(1)} $and $ ^f\!\cc^{(-1)}$, in the warped product $B\times F$. First step is to verify if $(^f\!\cc^{(1)},^f\!\cc^{(-1)})$ is a dual structure. This was already done in \cite{leonard}. 

\begin{proposition}
Let $B$ and $F$ two manifolds each one with a couple of affine connections $(\lB^{(1)},\lB^{(-1)})$ and $(\lF^{(1)},\lF^{(-1)})$ respectively.
\begin{itemize}
    \item If $(\lB^{(1)},\lB^{(-1)})$ and $(\lF^{(1)},\lF^{(-1)})$ are torsion free, then $\lf^{(1)}$ and $\lf^{(-1)}$ are also trosion free.
    \item If $(\lB^{(1)},\lB^{(-1)})$ and $(\lF^{(1)},\lF^{(-1)})$ are dual respect to $g_B$ and $g_F$ respectively, then $(\lf^{(1)},\lf^{(-1)})$ are dual with respect to $g_f$.
\end{itemize}
In particular, if $(B,g_B,\lB^{(1)},\lB^{(-1)})$ and $(F,g_F,\lF^{(1)},\lF^{(-1)})$ are statistical manifolds, then\\
$(B\times_f F,g_f,\lf^{(1)},\lf^{(-1)})$ is also statistical manifold.
\end{proposition}

For a complete demostration of these calims, we refer to \cite{olga} and reference therein. The second step is to lift the Amari-Chentsov tensor and verify that the lifting coincides with the Amari-Chentsov tensor of the liftings. 

\begin{theorem}\label{estructuraestadisticaWarped}

 The statistical structures $(B, g_{B}, \TB)$ and $(F, g_{F}, \TF)$ associated to  $(B,\lB^{(1)},\lB^{(-1)})$ and $(F, \lF^{(1)}, \lF^{(-1)})$ produce the statistical structure of the warped product:
\begin{equation}
    (M=B \times_{f} F, g_{f}, \TB \oplus \widetilde{f}^{2}{\TF} ).
\end{equation}
%con $\widetilde{f}=f \circ \pi \in \mathcal{F}(B\times F)$.
    
\end{theorem}

For the  proof we will use (consistently) the one-to-one relation between 3-tensor and dual structure as was presented in Theorem~\ref{thm:3ten=dualst}.

\begin{proof}

Consider  $X^{H},Y^{H},Z^{H}\in \mathfrak{L}(B)$ and $U^{V},V^{V},W^{V}\in \mathfrak{L}(F)$. We will apply the defining relation of the liftings of connections as in Proposition~\ref{prop:lift-conn}.
\begin{enumerate}
    \item The horizontal lifting is:
    \begin{align}
        \TB(X,Y,Z)^{H} & =g_{B}(\lB^{(-1)}_{X}Y-\lB^{(1)}_{X}Y,Z)^{H}= g_{B}((\lB^{(-1)}_{X}Y)^{H}-(\lB^{(1)}_{X}Y^{H}),Z^{H})=  \nonumber \\
        & = g_{f}(\lf^{(-1)}_{X^{H}}Y^{H}-\lf^{(1)}_{X^{H}}Y^{H},Z^{H})=\Tf(X^{H},Y^{H},Z^{H}).  \label{tensorB}
    \end{align}
    \item The vertical lifting is:
    \begin{align}
        \TF(U,V,W)^{V} & =g_{F}(\lF^{(-1)}_{U}V-\lF^{(1)}_{U}V,W)^{V} = g_{F}((\lF^{(-1)}_{U}V)^{V}-(\lF^{(1)}_{U}V)^{V}, W^{V})\nonumber \\
        & = \frac{1}{\widetilde{f}^{2}}g_{f}\Big((\lf^{(-1)}_{U^{V}}V^{V}+\frac{g_{f}(U^{V},V^{V})}{f}grad(f))-(\lf^{(1)}_{U^{V}}V^{V}+\frac{g_{f}(U^{V},V^{V})}{f} grad(f)), W^{V}\Big)\nonumber \\
        &=\frac{1}{\widetilde{f}^{2}}g_{f}\Big(\lf^{(-1)}_{U^{V}}V^{V}-\lf^{(1)}_{U^{V}}V^{V}, W^{V}\Big)=\frac{1}{\widetilde{f}^{2}} \Tf(U^{V},V^{V},W^{V}). \label{tensorF}
    \end{align}
    \item Finally note that the tensor in $X^{H}$ and $U^{V}$ vanishes because
    \begin{align}
        \Tf(X^{H}, U^{V}, \cdot) & =g_{f}(\lf^{(1)}_{X^{H}}U^{V}-\lf^{(-1)}_{X^{H}}U^{V}, \cdot) =g_{f}(\frac{1}{f}(Xf)U^{V}-\frac{1}{f}(Xf)U^{V}, \cdot)=0.\label{tensorBF}
    \end{align}
\end{enumerate}
From the results ~\eqref{tensorB} ~\eqref{tensorF} and ~\eqref{tensorBF}, and the uniqueness of the tensor, we get that in the warped product 
\begin{equation}\label{tensorwarped}
    \Tf=\TB \oplus \widetilde{f}^2{\TF}
\end{equation}
is the Amari-Chentsov tensor of the liftings. Therefore, the triple $(M,g_{f},\Tf)$ is a statistical manifold.  
\end{proof}

We now proceed to study the divergence under warped product. Note that the formulae for lifting of Amari-Chentsov tensor in \eqref{tensorwarped} motivate us to define a lifting of two divergence functions, i.e., if $D^B$ and $D^F$ are divergence functions on $B$ and $F$ respectively, we define its lifting as:
\begin{equation}\label{eq:warpedDiv}
    D^f=\pi^*D^B+(\tilde{f}_1)^2\sigma^*D^F
\end{equation}
 
 for $\tilde{f}_1(p,q)=\tilde{f}(p)$ for the lifting of $f$ to $M$. 

Before we give the claim on $D^f$, we need the following technical lemma: 
\begin{lemma}\label{lem5.4}
For any collection $\{V_j=(V_{B,j},V_{F,j}): j=1\dots,k\}$ and $\{W_j=(W_{B,j},W_{F,j}): j=1\dots,r\}$ of vector fields on $M=B\times F$ with respective decomposition in $TB$ and $TF$. We get:
$$ D^f(V_1\dots V_k\| W_1\dots W_r)=D^f(V_{B,1}\dots V_{B,k}\| W_{B,1}\dots W_{B,r})+\tilde{f}^2D^F(V_{F,1}\dots V_{F,k}\| W_{F,1}\dots W_{F,r}) $$
\end{lemma}

\begin{proof}
    For sake of simplicity, we first prove the claim for $k=1=r$. It is routine to verify that:
    \begin{align*}
        D^f(V\|W)&=V^1W^2(\pi^*D^B+(\tilde{f}_1)^2\sigma^*D^F)|_\Delta=\pi^*(V^1_BW^2_FD^B|_\Delta)+V^1\sigma^*((W_F^2\tilde{f}_1^2)D^F|_\Delta+\tilde{f}_1^2W_F^2D^F|_\Delta)\\
        &=\pi^*(V^1_BW^2_FD^B|_\Delta)+\sigma^*((V_F^1\tilde{f}_1^2)D^F|_\Delta+\tilde{f}_1^2V_F^1W_F^2D^F|_\Delta)=D^B(V_B\|W_B)+\tilde{f}^2D^F(V_F\|W_F),
    \end{align*}
   where the second equality comes from Leibniz rule by $W_F$ and the third one is Leibniz rule by $V_F$ by knowing that $W_F^2\tilde{f}_1^2=V_F^2\tilde{f}_1^2=0$.  
   
The case of $kr>1$, the sums also split in $B$-terms and $F$-terms, while the mixed terms vanish. Hence, the verification is similar but with more terms and not so easy  notation. 
\end{proof}

%rmk this verification works when V_j=V_{B,j}\oplus V_{F,j} as liftins of based vector fields...in general we should consider V_j=a_BV_{B,j}\oplus a_FV_{F,j} for a_B and a_F functions on M...but using that (aX)f=a(Xf) (by lie drivative if it is needed) led us to use the sma eargument to valid the claim

\begin{theorem}\label{thm:warpedDiv}%thm:warpedDivestructuraestadisticaWarped
    If $D^B$ and $D^F$ are divergences generating the metrics $g^B$ and $g^F$, and also the symmetric 3-tensor $T^B$ and $T^F$ respectively, then 
    \begin{enumerate}
    \item $D^f$ is a divergence function.
    \item $D^f$ generates the warped metric, i.e., $$ g_f(X,Y)=g^{(B)}(X_B,Y_B)+\tilde{f}^2g^{(F)}(X_F,Y_F).$$
\item The induced dual connections by $D^f$ is the warped product of the induced connections by $D^B,$ and $ D^F$.
 
  \item The induced tensor by $D^f$ is the warped product of the induced tensor by $D^B, D^F$, i.e., $$T^D=T^B+\tilde{f}^2T^F$$.
    \end{enumerate}
       
\end{theorem}

\begin{proof}
As conditions for divergence functions is done from inequalities, and inequalities are preserved by sum and product by no-negative functions, then $D^f$ in \eqref{eq:warpedDiv} is also a divergence function, and first claim follows.
    The proof of remaining claims is just a direct consequences of the previous lemma applied to the relations ~\eqref{metricdiver}-\eqref{T-diver}.
\end{proof}

Finally, despite of we have generated warped structure from warped product of divergences, this is not true for canonical divergence. For this, just recall the defining equation for canonical divergence in \eqref{divCanonica}, and note that 
for $\gamma_{p,q}$ the (unique) geodesic joining $p$ and $q$. We get: 
$${{g}}_{f}(V,W)=-VW\Big(\int_{0}^{1} t \| \overset{\cdot}{\gamma}^{f}_{(p_B,p_F),(q_B,q_F)}(t) \|^{2} dt\Big),$$
for $\gamma^f=(\gamma^B,\gamma^f)$ geodesic warped, and
$${{g}}_{f}(V,W)\neq -VW\Big(\int_{0}^{1} t\| \overset{\cdot}{\gamma}^B_{p_B,q_B}(t)\|^2dt+\tilde{f}^2\int_0^1 t\| \overset{\cdot}{\gamma}^F_{p_{F},q_{F}}(t)\|^2 dt \Big),$$
because we do not have that  $\gamma^B$ or $\gamma^F$ are geodesics. This yields that warped product of canonical divergences is not canonical divergence. also note that this implies that isostatical immersion does not carry enough information just because we can immerse warped product statistical manifolds but does not recover warped product of statistical model.

As a direct consequence of the previous fact, the warped product does not preserve most of the geometric structure coming from canonical divergences. In the particular case of information geometry, most of its main concepts %(like  Shannon entropy, Relative entropy, Fisher metric; etc {\color{red} and other entropy concepts see Kl-divergence in wikiepedia})
(like Shannon entropy, Fisher metric,
 expected log-likelihood ratio, information gain, relative entropy, among others) are not preserved by warped products, just because they are strongly related to the KL-divergence as particular case of canonical divergence. 

As conclusion, the geometry on the background of information geometry is preserved by warped products, but in the context of information theory this product is meaningless.

%\markboth{Introducción}{Referencias} % encabezado

\end{document}